\documentclass[11pt]{amsart}
\usepackage{amsmath,amsfonts,amssymb,amsthm}
\usepackage[utf8]{inputenc} 
\usepackage{color} 
\usepackage{tikz}
\usepackage{mathtools}
\usepackage{accents}
\usepackage{graphicx}
\usepackage{caption}
\usepackage{subfiles}
\providecommand{\abs}[1]{\lvert#1\rvert}
\newcommand{\defeq}{\vcentcolon=}
\newtheorem{lemma}{Lemma}[section]
\newtheorem{proposition}[lemma]{Proposition}
\newtheorem{theorem}[lemma]{Theorem}
\newtheorem{conjecture}[lemma]{Conjecture}
\newtheorem{corollary}[lemma]{Corollary}
\newtheorem{definition}[lemma]{Definition}

\newtheorem{remark}[lemma]{Remark}

\newtheorem{example}[lemma]{Example}
\usepackage{euscript} 
\let\matheu=\EuScript
\usepackage{mathrsfs}
\usepackage{hyperref}
\usepackage{tikz-cd}
\newcommand{\cl}{\text{cl}}
\newcommand{\FS}{\text{FS}}

\newcommand{\id}{\text{id}}
\newcommand{\Sym}{\text{Sym}}
\newcommand{\disk}{(\mathbb{D},\partial \mathbb{D})}
\newcommand{\spec}{\text{Spec}}
\newcommand{\Fuk}{\text{Fuk}}
\newcommand{\hol}{\text{hol}}

\newcommand{\tp}{\text{top}}
\newcommand{\coker}{\text{coker}}
\hypersetup{colorlinks=true,linkcolor=black,citecolor=black}
\tikzset{every picture/.style={line width=0.75pt}}
\numberwithin{equation}{section}
\title{An open GW-formula for Lagrangians in Fano varieties}
\author{Mohamed El Alami}
\begin{document}
\maketitle

\begin{abstract}
    Given a Fano variety $Y$ and a simple normal crossings divisor $D\subseteq Y$ which is anti-canonical, we prove a formula relating counts of discs with boundary on a Lagrangian $L\subseteq Y\backslash D$ to counts of rational curves in $Y$, under suitable \emph{positivity} assumptions on $L$. This formula seriously constrains the topology of $L$ in many examples. Our main application is a super-potential formula for Fano cyclic coverings $X$ of $Y$. As a corollary, we show that all the small components of the Fukaya category of a Fano hypersurface $X\subseteq \mathbb{P}^{n+1}$ are split-generated by monotone Lagrangian tori.
\end{abstract}
\tableofcontents
\section{Introduction}
In \cite{mypaper1}, we constructed a monotone Lagrangian torus $L$ in the index $1$ Fano hypersurface $X_{n+1}\subseteq \mathbb{P}^{n+1}$ and we showed that it split-generates the small component of the monotone Fukaya category $\text{Fuk}(X_{n+1})$. The torus $L$ was constructed as the pre-image of the Clifford torus $L_{\cl}\subseteq \mathbb{P}^n$ via a cyclic covering map $\phi:X_{n+1}\rightarrow \mathbb{P}^n$. In order to show split-generation, one needs to count Maslov index $2$ discs with boundary on $L$. This count was performed by means of a degeneration argument. Indeed, $X_{n+1}$ can be deformed to the singular toric hypersurface 
\begin{equation*}
X_{n+1}^0=V(t^{n+1}-x_0\dotsi x_n)\subseteq \mathbb{P}^{n+1},
\end{equation*}
where the counting problem is more explicit. 

In the present work, we consider a general $r$-fold cyclic covering map $\phi:X\rightarrow Y$ of Fano varieties which is branched along a smooth ample divisor $D_Y\subseteq Y$. Given a Lagrangian torus $L_Y\subseteq Y\backslash D_Y$, its pre-image $L_X\defeq \phi^{-1}(L_Y)\subseteq X$ will be Lagrangian for an appropriate choice of K\"ahler form. In all cases of interest, $L_X$ will also be a connected torus. Our aim is to relate the super-potentials of $L_X$ and $L_Y$.

Recall that the super-potential associated with $L_X$ is the formal sum
\begin{equation*}
 W_{L_X}= \sum_{\mu_{L_X}(\beta)=2} m_{0,\beta}(L_X)q^{\partial \beta},
\end{equation*}
where $m_{0,\beta}(L_X)$ is the count of discs in the class $\beta$ passing through a point $\text{pt}\in L_X$. If such a disc is disjoint from the ramification locus $D_X$, then it is just a lift of a Maslov index $2$ disc that is disjoint from the branch locus $D_Y$. However, if
a disc $u:\disk\rightarrow (X,L_X) $ of Maslov index $2$ intersects $D_X$ once at $z=0$, its image $v\defeq \phi\circ u:\disk \rightarrow (Y,L_Y)$ is a disc of Maslov index $2r$, with an $r$-fold tangency point to $D_Y$ at $z=0$. We show that this tangency point deforms to $r$ transverse intersection points at $t\zeta,\dots,t\zeta^r\in\mathbb{D}$, where $\zeta$ is a primitive $r^{\text{th}}$-root of unity, and $t$ is small. The idea then is to extend this deformation all the way to $t=1$, where $v$ breaks into smaller discs of Maslov index $2$ (see Figure \ref{intro-figure}). An interesting feature of this argument is the appearance of spherical bubbles as $t$ approaches $0$. These bubbles have a point constraint at $\infty$ and a divisor constraint at the $r^{\text{th}}$-roots of unity. Their count is $r!\langle\psi_{r-2}\text{pt} \rangle^Y_r$; the (regularized) point Gromov-Witten descendant in degree $r$.

To ensure compactness of the relevant moduli spaces of pseudo-holomorphic curves, we need to assume that $L_Y\subseteq Y\backslash D_Y$ is \emph{Maslov positive}, which means by definition that
\begin{equation*}
    \frac{1}{2}\mu_{L_Y}(v)\geq \max\{v\cdot D_Y,1\}
\end{equation*}
for all non-constant $I_Y$-holomorphic discs $v:\disk\rightarrow (Y,L_Y)$. This condition is stable under small perturbations of $I_Y$ which preserve $D_Y$, which is sufficient for transversality purposes.
\newpage
\begin{theorem}
Suppose the Lagrangian torus $L_Y\subseteq Y\backslash D_Y$ is Maslov positive. Then so is $L_X\subseteq X\backslash D_X$, and its super-potential is given by
\begin{equation}\label{super-potential-formula-introduction}
    \phi_*W_{L_X} = W^{Y\backslash D_Y}_{L_Y} + \left(W_{L_Y}^{D_Y}\right)^r - r!\langle\psi_{r-2}\emph{\text{pt}} \rangle^Y_r,
\end{equation}
where $\phi_*:\mathbb{C}[H_1(L_X)]\rightarrow \mathbb{C}[H_1(L_Y)]$ is pushforward on homology and $W_{L_Y}^{D_Y}$ (resp. $W^{Y\backslash D_Y}_{L_Y}$) is the contribution to $W_{L_Y}$ from discs which intersect (resp. do not intersect) $D_Y$.
\end{theorem}
\begin{figure}
\begin{tikzpicture}[x=0.75pt,y=0.75pt,yscale=-.8,xscale=.8]

\draw    (131,331) -- (591,329.5) ;
\draw  [draw opacity=0] (161.31,256.92) .. controls (187.55,259.79) and (206.84,269.38) .. (206.84,280.77) .. controls (206.84,294.43) and (179.12,305.5) .. (144.92,305.5) .. controls (110.72,305.5) and (83,294.43) .. (83,280.77) .. controls (83,269.28) and (102.64,259.61) .. (129.24,256.84) -- (144.92,280.77) -- cycle ; \draw   (161.31,256.92) .. controls (187.55,259.79) and (206.84,269.38) .. (206.84,280.77) .. controls (206.84,294.43) and (179.12,305.5) .. (144.92,305.5) .. controls (110.72,305.5) and (83,294.43) .. (83,280.77) .. controls (83,269.28) and (102.64,259.61) .. (129.24,256.84) ;  
\draw   (126.92,265.97) .. controls (126.92,255.22) and (134.98,246.5) .. (144.92,246.5) .. controls (154.86,246.5) and (162.92,255.22) .. (162.92,265.97) .. controls (162.92,276.73) and (154.86,285.44) .. (144.92,285.44) .. controls (134.98,285.44) and (126.92,276.73) .. (126.92,265.97) -- cycle ;
\draw  [dash pattern={on 4.5pt off 4.5pt}] (126.92,265.78) .. controls (126.92,262.34) and (134.98,259.55) .. (144.92,259.55) .. controls (154.86,259.55) and (162.92,262.34) .. (162.92,265.78) .. controls (162.92,269.22) and (154.86,272.01) .. (144.92,272.01) .. controls (134.98,272.01) and (126.92,269.22) .. (126.92,265.78) -- cycle ;
\draw  [color={rgb, 255:red, 0; green, 0; blue, 0 }  ,draw opacity=1 ][fill={rgb, 255:red, 0; green, 0; blue, 0 }  ,fill opacity=1 ] (142.76,283.11) -- (147.08,283.11) -- (147.08,287.78) -- (142.76,287.78) -- cycle ;
\draw  [color={rgb, 255:red, 1; green, 115; blue, 249 }  ,draw opacity=1 ][fill={rgb, 255:red, 2; green, 116; blue, 249 }  ,fill opacity=1 ] (159.87,265.98) -- (162.73,262.48) -- (165.97,265.58) -- (163.11,269.08) -- cycle ;
\draw  [color={rgb, 255:red, 1; green, 115; blue, 249 }  ,draw opacity=1 ][fill={rgb, 255:red, 3; green, 117; blue, 251 }  ,fill opacity=1 ] (136.11,271.33) -- (139.25,268.12) -- (142.21,271.52) -- (139.07,274.73) -- cycle ;
\draw  [color={rgb, 255:red, 1; green, 114; blue, 246 }  ,draw opacity=1 ][fill={rgb, 255:red, 1; green, 115; blue, 247 }  ,fill opacity=1 ] (133.95,261.1) -- (137,257.8) -- (140.05,261.88) -- (137,265.19) -- cycle ;
\draw   (276,242) .. controls (276,205.27) and (305.77,175.5) .. (342.5,175.5) .. controls (379.23,175.5) and (409,205.27) .. (409,242) .. controls (409,278.73) and (379.23,308.5) .. (342.5,308.5) .. controls (305.77,308.5) and (276,278.73) .. (276,242) -- cycle ;
\draw  [dash pattern={on 4.5pt off 4.5pt}]  (342.5,242) -- (360,242) -- (365,242) -- (409,242) ;
\draw  [dash pattern={on 4.5pt off 4.5pt}]  (297,193.5) -- (342.5,242) ;
\draw  [dash pattern={on 4.5pt off 4.5pt}]  (342.5,242) -- (299,291.5) ;
\draw   (521,230.25) .. controls (521,207.74) and (539.24,189.5) .. (561.75,189.5) .. controls (584.26,189.5) and (602.5,207.74) .. (602.5,230.25) .. controls (602.5,252.76) and (584.26,271) .. (561.75,271) .. controls (539.24,271) and (521,252.76) .. (521,230.25) -- cycle ;
\draw   (602.5,230.25) .. controls (602.5,216.44) and (613.69,205.25) .. (627.5,205.25) .. controls (641.31,205.25) and (652.5,216.44) .. (652.5,230.25) .. controls (652.5,244.06) and (641.31,255.25) .. (627.5,255.25) .. controls (613.69,255.25) and (602.5,244.06) .. (602.5,230.25) -- cycle ;
\draw   (491,183) .. controls (491,169.19) and (502.19,158) .. (516,158) .. controls (529.81,158) and (541,169.19) .. (541,183) .. controls (541,196.81) and (529.81,208) .. (516,208) .. controls (502.19,208) and (491,196.81) .. (491,183) -- cycle ;
\draw   (496,283) .. controls (496,269.19) and (507.19,258) .. (521,258) .. controls (534.81,258) and (546,269.19) .. (546,283) .. controls (546,296.81) and (534.81,308) .. (521,308) .. controls (507.19,308) and (496,296.81) .. (496,283) -- cycle ;
\draw  [color={rgb, 255:red, 0; green, 0; blue, 0 }  ,draw opacity=1 ][fill={rgb, 255:red, 0; green, 0; blue, 0 }  ,fill opacity=1 ] (536.68,261.43) -- (541,261.43) -- (541,266.11) -- (536.68,266.11) -- cycle ;
\draw  [color={rgb, 255:red, 0; green, 0; blue, 0 }  ,draw opacity=1 ][fill={rgb, 255:red, 0; green, 0; blue, 0 }  ,fill opacity=1 ] (531.68,198.43) -- (536,198.43) -- (536,203.11) -- (531.68,203.11) -- cycle ;
\draw  [color={rgb, 255:red, 0; green, 0; blue, 0 }  ,draw opacity=1 ][fill={rgb, 255:red, 0; green, 0; blue, 0 }  ,fill opacity=1 ] (600.34,227.91) -- (604.66,227.91) -- (604.66,232.59) -- (600.34,232.59) -- cycle ;
\draw  [color={rgb, 255:red, 4; green, 117; blue, 249 }  ,draw opacity=1 ][fill={rgb, 255:red, 2; green, 116; blue, 249 }  ,fill opacity=1 ] (512.95,183.2) -- (515.81,179.7) -- (519.05,182.8) -- (516.19,186.3) -- cycle ;
\draw  [color={rgb, 255:red, 1; green, 115; blue, 247 }  ,draw opacity=1 ][fill={rgb, 255:red, 1; green, 115; blue, 247 }  ,fill opacity=1 ] (517.95,283.2) -- (520.81,279.7) -- (524.05,282.8) -- (521.19,286.3) -- cycle ;
\draw  [color={rgb, 255:red, 2; green, 116; blue, 249 }  ,draw opacity=1 ][fill={rgb, 255:red, 4; green, 117; blue, 249 }  ,fill opacity=1 ] (624.45,230.45) -- (627.31,226.95) -- (630.55,230.05) -- (627.69,233.55) -- cycle ;
\draw  [color={rgb, 255:red, 4; green, 117; blue, 249 }  ,draw opacity=1 ][fill={rgb, 255:red, 2; green, 116; blue, 249 }  ,fill opacity=1 ] (361.95,242.2) -- (364.81,238.7) -- (368.05,241.8) -- (365.19,245.3) -- cycle ;
\draw  [color={rgb, 255:red, 4; green, 117; blue, 249 }  ,draw opacity=1 ][fill={rgb, 255:red, 2; green, 116; blue, 249 }  ,fill opacity=1 ] (324.95,258.2) -- (327.81,254.7) -- (331.05,257.8) -- (328.19,261.3) -- cycle ;
\draw  [color={rgb, 255:red, 4; green, 117; blue, 249 }  ,draw opacity=1 ][fill={rgb, 255:red, 2; green, 116; blue, 249 }  ,fill opacity=1 ] (325.95,228.2) -- (328.81,224.7) -- (332.05,227.8) -- (329.19,231.3) -- cycle ;
\draw  [fill={rgb, 255:red, 9; green, 0; blue, 0 }  ,fill opacity=1 ] (126.75,331) .. controls (126.75,328.65) and (128.65,326.75) .. (131,326.75) .. controls (133.35,326.75) and (135.25,328.65) .. (135.25,331) .. controls (135.25,333.35) and (133.35,335.25) .. (131,335.25) .. controls (128.65,335.25) and (126.75,333.35) .. (126.75,331) -- cycle ;
\draw  [fill={rgb, 255:red, 9; green, 0; blue, 0 }  ,fill opacity=1 ] (586.75,329.5) .. controls (586.75,327.15) and (588.65,325.25) .. (591,325.25) .. controls (593.35,325.25) and (595.25,327.15) .. (595.25,329.5) .. controls (595.25,331.85) and (593.35,333.75) .. (591,333.75) .. controls (588.65,333.75) and (586.75,331.85) .. (586.75,329.5) -- cycle ;
\draw  [color={rgb, 255:red, 208; green, 4; blue, 249 }  ,draw opacity=1 ][fill={rgb, 255:red, 206; green, 4; blue, 247 }  ,fill opacity=1 ] (584,260) -- (588,268.5) -- (580,268.5) -- cycle ;
\draw  [color={rgb, 255:red, 208; green, 4; blue, 249 }  ,draw opacity=1 ][fill={rgb, 255:red, 206; green, 4; blue, 247 }  ,fill opacity=1 ] (383,290) -- (387,298.5) -- (379,298.5) -- cycle ;
\draw  [color={rgb, 255:red, 208; green, 4; blue, 249 }  ,draw opacity=1 ][fill={rgb, 255:red, 206; green, 4; blue, 247 }  ,fill opacity=1 ] (184,295) -- (188,303.5) -- (180,303.5) -- cycle ;
\draw   (341.1,236.97) -- (343.9,246.03)(346.78,240.18) -- (338.22,242.82) ;
\draw   (142.31,242.99) -- (148.69,250.01)(148.81,243.49) -- (142.19,249.51) ;
\draw   (557.31,225.99) -- (563.69,233.01)(563.81,226.49) -- (557.19,232.51) ;

\draw (116,340) node [anchor=north west][inner sep=0.75pt]   [align=left] {$t=0$};
\draw (583,340) node [anchor=north west][inner sep=0.75pt]   [align=left] {$t=1$};
\draw (341,337) node [anchor=north west][inner sep=0.75pt]   [align=left] {$t$};
\end{tikzpicture}
\caption{}
    \label{intro-figure}
\end{figure}

This theorem explains how to recover Hori-Vafa Landau-Ginzburg models that are mirrors to Fano varieties, in many examples, as super-potentials associated with Lagrangian tori. One particular example is that of Fano hypersurfaces in projective spaces which have been studied by N.Sheridan in \cite{sheridan-fano}.
\begin{theorem}
The degree $d$ Fano hypersurface $X_d\subseteq\mathbb{P}^{n+1}$, with $d\leq n$, contains a monotone Lagrangian torus $L_d$. For each $\lambda\in\mathbb{C}^*$, there exists a $\mathbb{C}^*$-local system $\xi_{\lambda}$ such that $(L_d,\xi_\lambda)$ split-generates the component $\Fuk(X)_{\lambda}$ of the Fukaya category.
\end{theorem}
The Lagrangian tori $L_d$ only see the small components of the Fukaya category. In the critical case $d=n+1$, our construction still works: the monotone Lagrangian torus $L_{n+1}$ split-generates the small component. It also sees the big component, but it doesn't generate it (it behaves like a point in the mirror), see \cite{mypaper1} for a more detailed discussion. See also \cite{sheridan-fano} for a construction of Lagrangian spheres which generate the big component.

The super-potential formula (\ref{super-potential-formula-introduction}) is especially interesting when $D_Y$ is anti-canonical. Recall that a \emph{weak LG-model} for $Y$ is a Laurent polynomial $f_Y$ such that the constant term $c_0(f^k_Y)$ is the $k^{\text{th}}$-regularized quantum period of $Y$, for all $k\geq 2$ (see \cite{katzarkov-przyj-old-and-new}). In \cite{tonkonog-periods}, Tonkonog shows that super-potentials of monotone Lagrangian tori are weak LG-models. In light of this result, formula (\ref{super-potential-formula-introduction}) explains how to obtain a weak LG-model for $X$ given one for $Y$, when $D_Y$ is anti-canonical. This is a purely algebro-geometric statement.
\begin{conjecture}
Let $Y$ be a smooth Fano variety of index $r$ and $D_Y\subseteq Y$ a smooth anti-canonical divisor. Let $X$ be the $r$-fold cyclic covering of $Y$ which is branched along $D_Y$. If $f_Y$ is a weak LG-model for $Y$, then $f_Y^r-c_0(f^r_Y)$ is a weak LG-model for $X$.
\end{conjecture}

Our methods show for instance that this conjecture holds if $Y\backslash D_Y$ contains a graded exact Lagrangian torus, see also Remark \ref{correct mirror is quotient} and the discussion preceding it.

The deformation argument outlined above (see Figure \ref{intro-figure}) only requires $L_Y$ to be Maslov positive, oriented, spin, but not necessarily a torus. Unsurprisingly, one can exploit the spheres contributing to $\langle\psi_{r-2}\text{pt} \rangle^Y_r$ in order to produce enough circles in $L_Y$ that bound holomorphic discs in $Y$, hence constraining the topology of $L_Y$.
\begin{theorem}\label{topological constraint}
Suppose $Y$ is a Fano variety of index $r\geq 2$ such that $\langle \psi_{r-2}\emph{\text{pt}}\rangle^Y_{r}\neq 0$. Let $H_1,\dots,H_r$ be a collection of homologous divisors in general position whose union $D_Y=\cup_{i=1}^r H_i$ is anti-canonical. If $L_Y\subseteq Y\backslash D_Y$ is an oriented, spin, graded, and exact Lagrangian of non-positive sectional curvature, then $L_Y$ is finitely covered by a product $(S^1)^{r-1}\times K$. 
\end{theorem}

\textbf{Related works.} 
In \cite{tonkonog-higher-maslov-potentials}, Tonkonog uses SFT neck-stretching techniques to obtain a similar formula to (\ref{super-potential-formula-introduction}), involving certain structure constants of the symplectic cohomology $SH^*(Y\backslash D_Y)$ see \cite[Theorem 1.3]{tonkonog-higher-maslov-potentials}. The approach we follow uses fairly primitive techniques of $J$-holomorphic curve theory instead, it allows for divisors that are not anti-canonical (we only require $c_1(Y)\geq D_Y$), and Lagrangians that are not necessarily monotone. These generalizations are essential for our intended applications. We also note that the deformation argument we employ suggests an interpretation of Tonkonog's constants $c_{i,k}$ as relative Gromov-Witten invariants (see Figure \ref{intro-figure}). Making this connection rigorous, however, seems to require major technical work that goes beyond the scope of this article.

In \cite{logpss}, Ganatra and Pomerleano use their LogPSS map to explain how the non-vanishing of certain Gromov-Witten invariants of $Y$ with constraints on $D_Y$ yields a \emph{quasi-dilation} on $SH^*(Y\backslash D_Y)$, and hence (by Viterbo restriction), on the free loop space homology of any graded exact Lagrangian brane $L_Y\subseteq Y\backslash D_Y$. These quasi-dilations impose serious topological restriction when $L_Y$ has real dimension $3$. Our approach to Theorem \ref{topological constraint} is similar in spirit, but perhaps closer to Fukaya's work in 
\cite{Fukaya-icm}.

\textbf{Acknowledgements.} I want to thank Mark McLean and Nick Sheridan for numerous conversations about this project which have led to major improvements. I want to thank Paul Seidel for suggesting the deformation argument depicted in Figure \ref{intro-figure}, Aleksey Zinger for teaching me about SFT curves in relative GW-theory, Yank\i\ Lekili for the reference \cite{ritter-smith}, and Kenji Fukaya for explaining to me some of the content of \cite{FO3-book1}. This project was partially funded by ERC grant 850713 -- HMS.


\section{Discs with a global tangency}
\subsection{Preliminaries}
Let $(Z,I)$ be a smooth Fano variety of complex dimension $n\geq 2$ and let $\omega$ be a K\"ahler form on $Z$. We will sometimes call $(Z,I,\omega)$ a Fano K\"ahler triple. A divisor $D\subseteq\ Z$ is said to be \emph{simple normal crossings} if it is given as a union of smooth divisors $D = \cup_{i=1}^N D_i$ such that for any subset $I\subseteq  \{1,2,\dots,N\}$, the intersection $D_I = \bigcap_{i\in I} D_i$ is transverse. This means that near each point $p\in D_I$, there is a holomorphic chart with coordinates $(z_1,\dots,z_n)$ where
\begin{equation*}
    D_i=\{z_i=0\} \quad\hbox{for all}\ \ i\in I.
\end{equation*}

Let $D \subseteq Z$ be an anti-canonical simple normal crossings divisor and let $\Omega$ be a holomorphic volume form on $Z$ with poles along $D$.

For each closed oriented Lagrangian $L\subseteq Z\backslash D$, one can associate two Maslov classes $\mu_L\in H^2(Z,L)$ and $\eta_L\in H^1(L)$ which we now recall.
\begin{itemize}
    \item[-] \underline{The class $\mu_L$}: Given a map $u:\disk\rightarrow (Z,L)$, choose a generic section $s$ of the complex line bundle $\bigwedge^n_{\mathbb{C}} u^*TZ\rightarrow \mathbb{D}$ such that the restriction $s_{|\partial \mathbb{D}}$ agrees with the orientation of $\bigwedge^n_{\mathbb{R}}TL$. Then, $$\mu_L(u)\defeq 2\# s^{-1}(0).$$
    \item[-] \underline{The class $\eta_L$}: For each $p\in L$, there is a unique complex number $\gamma(p)\in\mathbb{C}^*$ such that $\gamma(p)^{-1}\Omega_p\in \bigwedge^n_{\mathbb{C}}T_pZ$ agrees with the orientation form on $\bigwedge^n_{\mathbb{R}}T_pL$. The class $\eta_L$ is twice the pullback of the generator $[d\theta]\in H^1(\mathbb{C}^*,\mathbb{Z})$ by the map $\gamma:L\rightarrow \mathbb{C}^*$,
    \begin{equation*}
        \eta_L = 2\gamma^*[d\theta].
    \end{equation*}
\end{itemize}
These two Maslov classes are related by the identity
\begin{equation}\label{Maslov identity}
    \mu_L(u) = 2u\cdot D + \eta_L(\partial u),
\end{equation}
which holds for all disc maps $u:\disk\rightarrow (Z,L).$

\begin{definition}
The Lagrangian manifold $L\subseteq Z\backslash D$ is \emph{admissible} if
\begin{equation}\label{admissibility-assumption}
    \langle \omega, \pi_2(Z\backslash D,L)\rangle = 0 \quad\hbox{and}\quad \eta_L=0.
\end{equation}
\end{definition}
From now on, we refer to $\mu_L$ as the Maslov class unless otherwise stated.
The notion of admissibility is flexible enough to include the following examples:
\begin{itemize} 
    \item[-] Let $Z$ be a toric Fano variety and $M:Z\rightarrow \Delta$ its moment map. Then $M^{-1}(\delta)$ is an admissible Lagrangian torus for all $\delta\in\text{int}(\Delta)$.
    \item[-] Objects of the relative Fukaya category $\text{Fuk}(Z,D)$ are admissible.
\end{itemize}
The purpose of admissibility is to ensure compactness for certain moduli spaces of holomorphic curves that we will construct later. 

Let $N$ be an open neighborhood of $D$ which is disjoint from $L$. For transversality purposes, we need to perturb the complex structure $I$ in the space $\mathscr{J}_N(Z,\omega)$ of almost complex structures $J$ on $Z$ which are $\omega$-compatible, and such that $J_{|N}=I_{|N}$. We fix $N$ for once and for all and we often simply say that
\begin{equation*}
    J = I \quad \quad \text{near}\  D.
\end{equation*}

For each Maslov index $2$ class $\beta\in H_2(Z,L)$, we may define a numerical invariant $m_{0,\beta}(L)\in\mathbb{Z}$ as follows.
Consider the moduli space
\begin{equation}\label{discs no constraint}
    \mathscr{M}(L,\beta)  = \{ v :\disk\rightarrow (Z,L) |\ \overline{\partial}_J v = 0, [v]=\beta\}.
\end{equation}
The pseudo-holomorphic discs in this moduli space are \`a-priori somewhere-injective (see \cite[Theorem A]{Lazzarini-simplediscs}), and therefore regular for a generic choice of $J\in\mathscr{J}_N(Z,\omega)$. Note that when $L$ is admissible, it does not bound non-constant $J$-holomorphic discs of index $0$ for any $J\in \mathscr{J}_N(Z,\omega)$. With that in mind, the integer $m_{0,\beta}(L)$ is the degree of the (pseudo-cycle) evaluation map
\begin{equation*}
    ev:\mathscr{M}_{0,1}(L,\beta)\defeq \mathscr{M}(L,\beta)\times\partial \mathbb{D}/\text{Aut}(\mathbb{D})\rightarrow L.
\end{equation*}
These numerical invariants are often conveniently packaged in a polynomial
\begin{equation*}
    W_L\in \mathbb{C}[H_2(Z,L)]\defeq\left\{a_1q^{\beta_1}+\dotsi+ a_mq^{\beta_m}| a_i\in\mathbb{C},\ \beta_i\in H_2(Z,L)  \right\},
\end{equation*}
where $q$ is a formal parameter.
\begin{definition}\label{super-potential-definition}
The super-potential associated with $L$ is the polynomial
\begin{equation*}
    W_L = \sum_{\mu_L(\beta)=2} m_{0,\beta}(L) q^{\beta},
\end{equation*}
where $q$ is the formal parameter of the ring $\mathbb{C}[H_2(Z,L)]$.
\end{definition}
If $P\in \mathbb{C}[H_2(Z,L)]$ is a polynomial, and $\alpha\in H_2(Z,L)$ is a homology class, we denote by $P[\alpha]$ the coefficient of $P$ in degree $\alpha$.

\subsection{Transversality}
In order to achieve transversality for discs with higher Maslov indices, we need to perturb the $J$-holomorphic equation using domain-dependent almost complex structures $K=(J_z)\in \mathscr{K}_J$,
\begin{equation}\label{perturbation datum}
  \mathscr{K}_J = \{J_z\in \mathscr{J}_N(Z,\omega), z\in \mathbb{D}\ |\ J_z = J \hbox{ for all } z\in \{0\}\cup\partial\mathbb{D}\}. 
\end{equation}
The perturbed $\overline{\partial}$-equation for disc maps $u:\disk\rightarrow (Z,L)$ is then
\begin{equation}\label{pseudo-holomorphic-equation}
    \partial_s u + J_{s,t}(u)\partial_t u =0,
\end{equation}
where $(s,t)$ are the real and imaginary parts of the holomorphic coordinate $z\in\mathbb{D}$. Observe that near the divisor $D$, the equation above reduces to an honest $I$-holomorphic equation. The choice of a pair $(J,K)$, where $K= (J_z)\in\mathscr{K}_J$, is called a perturbation datum. For each such datum, non-constant solutions of  (\ref{pseudo-holomorphic-equation}) have energy
\begin{equation*}
    E(u)\defeq\frac{1}{2}\int_{\mathbb{D}}\abs{du}^2_{J_{s,t}}ds\wedge dt  = \int_{\mathbb{D}} u^*\omega > 0.
\end{equation*}

\begin{definition}
Let $l\geq 0$ be an integer. A disc map $u:\disk \rightarrow (Z,L)$ is said to be tangent to $D$ to order $l$ at $0\in\mathbb{D}$ if, near $u(0)\in Z$, the divisor $D$ is the (reduced) zero locus $(f=0)$ of a holomorphic function $f$ such that
\begin{equation*}
    f(u(z)) = O(z^l).
\end{equation*}
When this holds, we use the notation
\begin{equation*}
j^D_{0,l-1}(u)=0.    
\end{equation*}
\end{definition}

For each $l\geq 0$, define the tangency moduli space
\begin{equation}\label{tangency moduli space}
\mathscr{T}^D_{l} = \{u:\disk\rightarrow (Z,L) \ |\ \partial_s u + J_{s,t}(u)\partial_t u = 0,\ j^D_{0,l-1}(u)=0\}.
\end{equation}
For example, $\mathscr{T}^D_{0}$ is the space of unconstrained pseudo-holomorphic discs with boundary on $L$.
We will later restrict to a fixed homology class $\alpha\in H_2(Z,L)$ and work with the subspace
\begin{equation*}
 \mathscr{T}^D_{l}(\alpha) = \{u\in \mathscr{T}^D_{l} \ |\ [u]=\alpha\}.
\end{equation*}
\begin{remark}\label{remark about notation}
It would be more correct if we denoted the space  (\ref{tangency moduli space}) above by $\mathscr{T}_l(Z,L,J,(J_z))$. For the sake of clarity, we avoid this notation when there is little risk of confusion.
\end{remark}

Suppose for a moment that $D$ is smooth. Then, the jet maps $j^D_{0,l}$ have a simple geometric interpretation which we now explain. Let $\sigma$ be a section of the line bundle $\mathscr{O}_Z(D)$, whose zero locus is $D$. The jet map $j^D_{0,0}$ should be thought of as a section of the the line bundle
\begin{equation}
ev_0^*\mathscr{O}_Z(D)\rightarrow \mathscr{T}^D_0: \quad   j^D_{0,0}=ev_0^*\sigma,
\end{equation}
where $ev_0:\mathscr{T}^D_0\rightarrow Z$ is evaluation at $0\in\mathbb{D}$. The zero locus of this section is $\mathscr{T}^D_1$ as defined above. The restriction of the line bundle $ev_0^*\mathscr{O}_Z(D)$ to $\mathscr{T}^D_1$ is the normal line bundle $N_{D/Z}$, i.e.
\begin{equation}\label{normals}
    (ev_0^*\mathscr{O}_Z(D))_u = T_{u(0)}Z/T_{u(0)}D,
\end{equation}
for all $u\in \mathscr{T}^D_1$. The jet map $j^D_{0,1}$ is a section of the line bundle in (\ref{normals}). Indeed, in a neighborhood $U$ of $u(0)\in Z$, the section $\sigma$ gives a local defining equation $f: U\rightarrow \mathbb{C}$ for the divisor $D$. The function $f\circ u : \mathbb{D}\rightarrow\mathbb{C}$ has an analytic expansion near $0\in\mathbb{D}$,
\begin{equation*}
    f(u(z)) = \lambda z + o(z).
\end{equation*}
The $1$-jet map at $u$ may be interpreted as the normal vector
\begin{equation}\label{jet map formula}
    j^D_{0,1}(u)=(df)_{u(0)}^{-1}(\lambda) \in N_{D/Z,u(0)}.
\end{equation}
The later is independent of the choice of a holomorphic (local) defining equation for the smooth divisor $D$. The higher jets can be interpreted as sections
\begin{equation*}
   j^D_{0,l} : \mathscr{T}^D_{l} \rightarrow N_{D/Z} 
\end{equation*}
in exactly the same way.

Following the general strategy outlined in \cite[(9k)]{PL}, transversality for the spaces $\mathscr{T}^D_{l}$ can be achieved using a generic domain dependent perturbation datum $(J_z)
\in \mathscr{K}_J$. The complications of the tangency constraint are resolved in the work of Cieleback and Mohnke in \cite[\S 6]{CM-transversality}, where it is shown inductively that for each $l$, there is a comeagre set of perturbation data $(J_z)$ for which $\mathscr{T}^D_{l}$ is smooth. Moreover, for each $u\in \mathscr{T}^D_{l} $, the vertical derivative
\begin{equation*}
     d_uj^D_{0,l} : T_u\mathscr{T}^D_{l} \rightarrow \mathbb{C} 
\end{equation*}
is surjective, see in particular \cite[Proposition 6.9]{CM-transversality}. The main difference in our setup is that we allow domain dependent perturbations, so we need not restrict to somewhere-injective discs for transversality purposes.

When $D$ is simple normal crossings with components $D_i$, the tangency order $l$ has a contribution $l_i\geq 0$ from each divisor $D_i$ which can all be organized in a \emph{tangency vector} $\mathbf{v}=(l_1,\dots,l_N)$, such that $\sigma(\mathbf{v})\defeq l_1+\dots+l_n=l$. For each tangency vector $\mathbf{v}$, set 
\begin{align*}
    \mathscr{T}_{\mathbf{v}} = \left\{u\in\mathscr{T}^D_l \ \ \big{|}\ \  j^{D_i}_{0,l_i-1}(u)=0\ \hbox{for}\ 0\leq i \leq N\right\}
    = \bigcap_{i=1}^N \mathscr{T}^{D_i}_{l_i}.
\end{align*}

Note that one has an inclusion $\mathscr{T}_{\mathbf{v}}\rightarrow \mathscr{T}_{\mathbf{v}'}$ whenever $\mathbf{v}\leq \mathbf{v}'$.
Transversality for these moduli spaces is achieved inductively on the multiplicity vector $\mathbf{v}$ using the same methods of \cite[\S 6]{CM-transversality}, see also \cite[Lemma 4.15]{logpss}.
\begin{lemma}\label{transversality}
For each $J\in \mathscr{J}_N(Z,\omega)$, there is a comeagre set of domain-dependent perturbations $(J_z)\in \mathscr{K}_{J}$ for which all the spaces $\mathscr{T}_{\mathbf{v}}(\alpha)$ are smooth manifolds. The associated dimension is given by the Riemann-Roch formula
\begin{equation*}
    \dim\mathscr{T}_{\mathbf{v}}(\alpha) = n+\mu_L(\alpha) - 2\sigma(\mathbf{v}).
\end{equation*}
\end{lemma}
The previous transversality result can be understood concretely as follows. Let $\mathbf{v}=(l_1,\dots,l_N)$ be a tangency vector, and let $\mathbf{v}'=(l_1+1,l_2,\dots,l_N)$. For each $u_0\in \mathscr{T}_{\mathbf{v}'}\subseteq \mathscr{T}_{\mathbf{v}}$, the divisors $D_i$ are cut out by equations $\{f_i=0\}$ locally near $u_0(0)\in Z$ (some $f_i$'s may be invertible in this local chart). Assuming $\mathscr{T}_{\mathbf{v}}$ is smooth, there is a well defined smooth map in a neighborhood of $u_0\in \mathscr{T}_{\mathbf{v}}$ given by
\begin{equation*}
   d^{l_1}: \mathscr{T}_{\mathbf{v}} \dashrightarrow \mathbb{C},\quad
    u \mapsto \frac{d^{l_1}f_1(u)}{dz^{l_1}}\bigg|_{z=0}.
\end{equation*}
The content of the previous transversality lemma is that $0\in\mathbb{C}$ is a regular value of the map $d^{l_1}$. The zero-set $(d^{l_1})^{-1}(0)$ therefore provides a local chart for $\mathscr{T}_{\mathbf{v}'}$ near $u_0$. Inductively, one sees in fact that $0\in\mathbb{C}^{\sigma(\mathbf{v})}$ is a regular value of the map $j_{0,\mathbf{v}}: \mathscr{T}_{0} \dashrightarrow \mathbb{C}^{\sigma(\mathbf{v})}$ given by
\begin{equation}\label{big jet function}
    j_{0,\mathbf{v}}(u) = \left(j_{0,l_1-1}(f_1(u)),\dots,j_{0,l_N-1}(f_N(u))\right),
\end{equation}
where we have used the notation
\begin{equation*}
    j_{0,d-1}(h)=(h(0),h'(0),\dots,h^{(d-1)}(0))\in\mathbb{C}^d
\end{equation*}
for functions $h:\mathbb{D}\rightarrow \mathbb{C}$ which are holomorphic near $0$.
\subsection{Unfolding tangency points}
We now explain how a disc which is maximally tangent to $D$ to order $r$ can be deformed to a nearby disc with $r$ transverse intersection points with $D$.
Let $\alpha \in H_2(Z,L)$ be an integral class for which
\begin{equation}\label{constraint on alpha}
    r\defeq\frac{1}{2}\mu_L(\alpha) = \alpha\cdot D\geq 2.
\end{equation}
We associate with $\alpha$ the tangency vector
\begin{equation*}
    \mathbf{v}= (\alpha\cdot D_1,\dots,\alpha\cdot D_N).
\end{equation*}
Let $s_i\in H^0(Z,\mathscr{O}_Z(D_i))$ be a holomorphic section whose zero locus is $D_i$. For each $u\in \mathscr{T}_0(\alpha)$, we can pullback the pair $(\mathscr{O}_Z(D_i),s_i)$ in order to produce a holomorphic line bundle $u^*\mathscr{O}_Z(D_i)$ over $\mathbb{D}$, with a \emph{smooth} section $s_{i,u} = u^*s_i$ that satisfies the following properties:
\begin{itemize}
    \item[-] The section $s_{i,u}$ is nowhere vanishing along the boundary $\partial\mathbb{D}$.
    \item[-] Near its zeroes, $s_{i,u}$ is holomorphic. 
    \item[-] The count $\#s_{i,u}^{-1}(0)$ of zeroes with multiplicity is $r_i\defeq\alpha\cdot D_i$.
\end{itemize}
As a consequence, for each $i=1,\dots,N$, we have a globally defined map
\begin{equation}\label{Phi map components}
    \Phi_{i,\alpha}: \mathscr{T}_0(\alpha) \rightarrow \Sym^{r_i}(\mathbb{D}),\quad
    u\mapsto s_{i,u}^{-1}(0).
\end{equation}
These maps are smooth because of the argument principle, see equation (\ref{argument principle}) below. These maps are also the component of
\begin{equation}\label{Phi map}
    \Phi_{\alpha}:\mathscr{T}_0(\alpha) \rightarrow \Sym^{\mathbf{v}}(\mathbb{D}),
\end{equation}
where we have used the notation $\mathbf{v}=(r_1,\dots,r_N)$ and 
\begin{equation*}
    \Sym^{\mathbf{v}}(\mathbb{D})\defeq \Sym^{r_1}(\mathbb{D})\times\dots\times \Sym^{r_N}(\mathbb{D}).
\end{equation*}

\begin{lemma}\label{transversality near 0}
The map $\Phi_{\alpha}$ is regular above $0\in\Sym^{\mathbf{v}}(\mathbb{D})$.
\end{lemma}
\begin{proof}
Recall that near each $u_0\in\Phi_{\alpha}^{-1}(0)=\mathscr{T}_{\mathbf{v}}(\alpha)\subseteq \mathscr{T}_0(\alpha)$, there is a locally defined smooth function $j_{0,\mathbf{v}}: \mathscr{T}_{0}(\alpha) \dashrightarrow \mathbb{C}^{\sigma(\mathbf{v})}$ whose zero locus provides a smooth local chart for $\Phi_{\alpha}^{-1}(0)$, see (\ref{big jet function}). Moreover, there is a local diffeomorphism $j_{\mathbf{v}}: \Sym^{\mathbf{v}}(\mathbb{D})\rightarrow \mathbb{C}^{\sigma(\mathbf{v})}$ near $0$. It is given by components $j_{r_i}: \Sym^{r_i}(\mathbb{D})\rightarrow \mathbb{C}^{r_i}$ which are defined by
\begin{equation*}
    j_{r_i}([(a_1,\dots,a_{r_i})])= j_{0,r_i-1}\left(\prod_{j=1}^{r_i}\frac{z-a_j}{1-\overline{a}_{j}z} \right),
\end{equation*}
see Lemma 3.9 of \cite{mypaper1} for a proof. We will show that the derivatives $d_{u_0}j_{0,\mathbf{v}}$ and $d_{u_0}(j_{\mathbf{v}}\circ\Phi_{\alpha})$ are related by an invertible matrix. Since the former is surjective, then so is the latter. 

Let $u_t\in \mathscr{T}_0({\alpha})$ be a smooth path of pseudo-holomorphic discs through $u_0$. Let $f_i$ be a local defining equation for $D_i$ near $u_0$. Then for each $t$, there is a unique product decomposition
\begin{equation*}
    f_i(u_t(z)) = m^i_t(z)\times g_i(t,z),
\end{equation*}
where $g_i(t,z)$ is nowhere vanishing, and
\begin{equation*}
    m^i_t(z) = \prod_{j=1}^{r_i}\frac{z-a_j(t)}{1-\overline{a}_j(t)z}
\end{equation*}
is product of M\"obius transformations. By the argument principle,
\begin{equation} \label{argument principle}
    \sum_{j=1}^{r_i} a_j(t)^k = \frac{1}{2\pi i}\oint_{\gamma} z^k\frac{(f\circ u_t)'(z)}{f(u_t(z))} dz,
\end{equation}
 where $\gamma$ is the sum the boundaries of small discs around each $a_j(t)$. It follows that the path $m^i_t(z)$ is smooth in $t$. Finally, since $m_0^i(z)=z^{r_i}$,
 \begin{align*}
    \frac{d}{dt}\Bigg|_{t=0} j_{0,r_i-1}\left( f(u_t(z)) \right) 
    &= j_{0,r_i-1}\left(\frac{d}{dt}\Bigg|_{t=0}m^i_t(z) g_i(0,z)\right)\\
    &= j_{0,r_i-1}\left(\frac{d}{dt}\Bigg|_{t=0}m^i_t(z)\right)M(g_i),
\end{align*}
where $M(g_i)$ is an upper triangular matrix whose diagonal elements are all equal to $g_i(0,0)\neq 0$. Combining this computation for all $i=1,\dots,N$, we sees that $d_{u_0}j_{0,\mathbf{v}}$ and $d_{u_0}(j_{\mathbf{v}}\circ\Phi_{\alpha})$ are related by an upper triangular matrix with non-zero elements on its diagonal, hence invertible.
\end{proof}

In contrast with the previous lemma, when $u$ has transverse intersections with $D$, its regularity can be checked by means of classical transversality methods, i.e. the regularity of an appropriate evaluation map.

\begin{lemma}\label{transversality by evaluation}
Let $\mathbf{z}=(z_i)\in \mathbb{D}^r$ be an ordered collection of pair-wise distinct interior points. Then $[\mathbf{z}]\in \Sym^{\mathbf{v}}\mathbb{D}$ is regular for $\Phi_{\alpha}$ if and only if the evaluation map
\begin{equation*}
    ev_\mathbf{z}: \mathscr{T}_0(\alpha)\rightarrow Z^r
\end{equation*}
is transverse to $D^r\subseteq Z^r$.
\end{lemma}
\begin{proof}
Note that any $u_0\in \Phi_{\alpha}^{-1}([\mathbf{z}])\subseteq \mathscr{T}_0(\alpha)$ intersects the divisor $D$ in its smooth locus, so the statement of the lemma makes sense.
We choose a local (near each $u_0(z_i)$) defining equation $f$ for the divisor $D$. We can again compare the two derivatives
\begin{align}
    d_{u_0}\Phi_{\alpha}: T_{u_0}\mathscr{T}_0(\alpha) &\rightarrow T_{[\mathbf{z}]} \Sym^r\mathbb{D} = \bigoplus_{i=1}^r T_{z_i}\mathbb{D},\label{two-derivatives-1}\\
    d_{u_0} ev_\mathbf{z}: T_{u_0}\mathscr{T}_0(\alpha) &\rightarrow \bigoplus_{i=1}^r T_{u_0(z_i)}Z/T_{u_0(z_i)}D \xrightarrow[]{df} \mathbb{C}^r.\label{two-derivatives-2}
\end{align}
Let $u_s$ be a small deformation of $u_0$. Following the same ideas of Lemma \ref{transversality near 0}, there is factorization
\begin{equation*}
    f(u_t(z)) = g_t(z)\prod_{i=1}^r (z-z_i(t))
\end{equation*}
which is smooth in $t$, and such that each $g_t$ is nowhere vanishing. By direct computation,
\begin{equation*}
    df\left(\frac{d}{dt}\Bigg|_{t=0} u_t(z_i)\right) = (-1)^r z_i'(0)\prod_{j\neq i} (z_j-z_i) g_0(z_i).
\end{equation*}
It follows that the two derivatives (\ref{two-derivatives-1}) and (\ref{two-derivatives-2}) are related by an invertible diagonal matrix.
\end{proof}

\subsection{Compactness}
To simplify the study of compactness (see also Remark \ref{why fix p and chi} below), we fix an angle $\chi\in \partial\mathbb{D}$ which is distinct from all roots of unity. We emphasize that all of our discs are stabilized by the parametrization given to $\mathbb{D}\subseteq \mathbb{C}$. An equivalent interpretation is to view them stabilized by the choice of the pair $(0,\chi)\in (\mathbb{D},\partial \mathbb{D})$. 

Let $z_1=(\zeta,\zeta^2,\dots,\zeta^r)\in\mathbb{D}^{\sigma(\mathbf{v})}$ be the ordered collection of $r^{\text{th}}$-roots of unity, and we denote by $[z_1]$ the corresponding element in $\Sym^{\mathbf{v}}(\mathbb{D})$.
We think of $z_1$ as the product $\boldsymbol{\zeta}_1\times\dotsi\times \boldsymbol{\zeta}_N$ of $N$ vectors, where $\boldsymbol{\zeta}_i\in \mathbb{D}^{r_i}$.

We fix an almost complex perturbation $J\in \mathscr{J}_N(Z,\omega)$ such that the following moduli spaces 
\begin{gather*}
  \mathscr{T}_r^{s}(Z,D,p) = \{\hbox{simple } u: \mathbb{P}^1 \rightarrow Z|\ \overline{\partial}_J u = 0,\ u(\infty)=p,\ j_{0,r-1}^D u = 0 \},\\
   \mathscr{M}^{s}_{\mathbf{z}_1}(Z,D,p)= \{\hbox{simple }u:\mathbb{P}^1 \rightarrow Z|\ \overline{\partial}_J u = 0,\ u(\infty)=p,\ u(\boldsymbol{\zeta_i})\subseteq D_i\},\\
   \mathscr{M}^{\text{s.i}} = \{u:\disk\rightarrow (Z,L)\ |\ \overline{\partial}_J u = 0,\ u \hbox{ is somewhere-injective}\}
\end{gather*}
are all Fredholm regular. We refer to \cite[\S 4.4]{wendl-notes} and \cite[Lemma 6.7]{CM-transversality} for a proof that such $J$ exists. Next, we choose a domain-dependent $(J_z)\in\mathscr{K}_{J}$ such that the moduli spaces $\mathscr{T}_{\mathbf{v}}(\alpha)$ are also Fredholm regular. We also fix a point $p\in L$ that is transverse to all the (countably many) evaluation maps of discs at the boundary point $\chi$, or at the roots of unity. In particular, we have regularity for the spaces
\begin{equation}\label{tangency with homology class}
    \mathscr{T}_{\mathbf{v}}^{\chi}(\alpha) = \{ u \in \mathscr{T}_\mathbf{v}(\alpha)  \ | \ u(\chi) = p \}
\end{equation}
for all tangency vector $\mathbf{v}$. 

The study of compactness is fairly tricky. The main issue is the appearance of disc bubbles that are (\`a-priori) not regular, or the appearance of spherical bubbles that sink into the divisor $D$. While there are techniques in the literature that deal with both of these issues, we choose to bypass them entirely as they do not arise in our intended applications. From now on, we assume that $L\subseteq Z\backslash D$ is admissible.

\begin{definition}
The pseudo-index $j_Z$ of a Fano variety $(Z,I)$ is the smallest Chern number of an $I$-holomorphic map $u:\mathbb{P}^1\rightarrow Z$. Alternatively,
\begin{equation*}
    j_Z= \min\{-C\cdot K_Z \ | \ C\subseteq Z\ \hbox{ an algebraic curve}\}.
\end{equation*}
\end{definition}

\begin{lemma}\label{no disc bubbles}
 Let $u_k\in\mathscr{T}^{\chi}_0(\alpha)$ be a sequence of discs such that $\Phi_{\alpha}(u_k)$ converges in the interior of $\Sym^{\mathbf{v}}(\mathbb{D})$ (see (\ref{Phi map})). Then, the sequence $(u_k)$ does not exhibit disc bubbles.
\end{lemma}
\begin{proof}
We follow the same arguments that appear in \cite[\S 4.6]{mcduff-salamon}, especially the proof of Lemma 4.6.5. We fix an identification $\mathbb{D}\simeq \mathbb{H}$ to keep our notation consistent with this reference.

If a disc bubble arises, then there is a sequence $\xi_k\in \text{int}(\mathbb{H})$ converging to $\xi\in\partial \mathbb{H}$ such that 
$c_k \defeq \abs{du_k(\xi_k)} \rightarrow \infty$. The sequence $(\xi_k)$ is constructed from Hofer's Lemma (4.6.4 in \cite{mcduff-salamon}), which also provides a sequence $\epsilon_k\rightarrow 0$ such that $\epsilon_k c_k\rightarrow \infty$. The disc bubble $v_{\infty}$ then arises as the uniform $C^{\infty}$-limit of the re-scaled pseudo-holomorphic maps
\begin{equation*}
    v_k(z) = u_k\left(\xi_k +\frac{z}{c_k}\right) \quad \text{defined for}\ z\in B_{\epsilon_k c_k}(\xi_k)\cap\mathbb{H}\subseteq \mathbb{H}.
\end{equation*}

Because $L\cap D =\emptyset$, the disc bubble $v_{\infty}$ cannot be entirely inside of $D$. Since $\Phi_{\alpha}(u_k)$ converges in the interior of $\Sym^{\mathbf{v}}(\mathbb{D})$, the maps $v_k$ have no intersection with $D$ when $k$ is sufficiently large. Hence, the limit $v_{\infty}$ has no intersection with $D$ either (this is an application of the argument principle to $f\circ v_k$, where $f$ is a local defining equation for $D$). We deduce that
\begin{equation}\label{contradiction 1}
    v_{\infty}\cdot D = 0,
\end{equation}
which contradictions the admissibility assumption \eqref{admissibility-assumption} on $L$.
\end{proof}

\begin{lemma}\label{proper map}
 Suppose that $r\leq j_Z$. Let $\mathbf{z}\in\mathbb{D}^{r}$ be a vector whose components are pair-wise distinct. Then, the fiber $\Phi_{\alpha}^{-1}([\mathbf{z}])$ of the map
 \begin{equation}\label{Phi map 2}
 \Phi_{\alpha}:\mathscr{T}_0^{\chi}(\alpha)\rightarrow\text{Sym}^{\mathbf{v}}(\mathbb{D})
 \end{equation}
 is compact.
\end{lemma}
\begin{proof}
Disc bubbles are excluded by Lemma \ref{no disc bubbles}. If a sequence $(u_k)\in \Phi_{\alpha}^{-1}([\mathbf{z}])$ is not $C^1$-bounded, then a spherical bubble occurs and, because of the assumption $r\leq j_Z$, it must be the only non-constant component of the Gromov limit. But elements of $\Phi_{\alpha}^{-1}([\mathbf{z}])$ have $r\geq 2$ points constrained to $D$. In particular, the fundamental component can't be constant and that's a contradiction.
\end{proof}

\begin{lemma}\label{compactness for tangency discs}
 Suppose $r\leq j_Z$. Then, the moduli space $\mathscr{T}_{\mathbf{v}}^{\chi}(\alpha)$ is compact.
\end{lemma}
\begin{proof}
We use a similar argument to the proof of Lemma \ref{proper map}. If a sequence $(u_k)\in\mathscr{T}_{\mathbf{v}}^{\chi}(\alpha)$ is not $C^1$-bounded, then its Gromov limit is a constant disc $\mathscr{M}_{\chi}$ attached at $0\in\mathbb{D}$ to a spherical bubble 
\begin{equation*}
    u_{\infty} \in \{ u\in \mathscr{T}_r^{s}(Z,D,p)\ |\ [u]=\alpha\}.
\end{equation*}
Such a rational curve is automatically simple (it has minimal Chern number) and hence regular. However, it belongs to a regular $0$-dimensional moduli space that carries a non-trivial $\mathbb{C}^*$-action, so it cannot exist.
\end{proof}

\begin{definition}
The tangency number $\tau^D_{\alpha}(L)$ is the signed count of the elements in the moduli space $\mathscr{T}_{\mathbf{v}}^{\chi}(\alpha)$.
\end{definition}
Just like the invariants $m_{0,\beta}(L)$, the tangency numbers $\tau^D_{\alpha}(L)$ do not depend on the choice of a generic $J\in\mathscr{J}_N(Z,\omega)$, or the perturbation datum $(J_z)\in \mathscr{K}_J$. This is because all non-constant pseudo-holomorphic discs on $L$ have a positive Maslov number.

\begin{example}
In \cite{mypaper1}, we compute the tangency number $\tau^D_{\alpha}(L)$ for the Clifford torus $L_{\cl}\subseteq \mathbb{P}^2$ relative to the toric boundary $\Delta=\{x_0x_1x_2=0\}$, and also relative to an elliptic curve $E$ obtained by smoothing the corners of $\Delta$. For instance, there is a unique class $\alpha\in H_2(Z,L)$ with the intersection numbers
\begin{equation*}
    \alpha\cdot\{x_0=0\}=0,\ \alpha\cdot\{x_1=0\}=1,\  \text{and}\ \alpha\cdot\{x_2=0\}=2.
\end{equation*}
This class has the following tangency numbers:
\begin{equation*}
    \tau_{\alpha}^{\Delta}(L_{\text{cl}})=1,\ \text{and}\ \tau_{\alpha}^{E}(L_{\text{cl}})=3.
\end{equation*}
\end{example}

\begin{remark}\label{why fix p and chi}
Although the space $\mathscr{T}_{\mathbf{v}}(\alpha)$ (without the constraint $u(\chi)=p$) may not be compact, it can be compactified by a codimension $2$ set of spheres attached to constant discs as in the proof of Lemma \ref{compactness for tangency discs}. In the language of \cite{zinger-pseudocycles}, $ev_{\chi}:\mathscr{T}_{\mathbf{v}}(\alpha)\rightarrow L$ is a pseudo-cycle and $\tau^D_{\alpha}(L)$ is its degree.
\end{remark}

Despite being proper above $0\in \text{Sym}^{\mathbf{v}}(\mathbb{D})$, the map $\Phi_{\alpha}$ from (\ref{Phi map 2}) need not be proper near $0$.

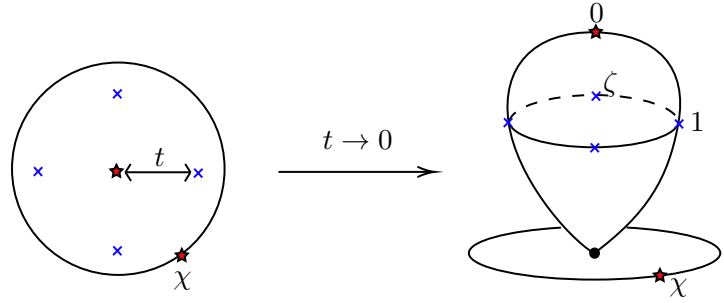
\begin{figure}
    \centering
    \begin{tikzpicture}[x=0.75pt,y=0.75pt,yscale=-.8,xscale=.8]

\draw  [draw opacity=0] (421.44,236.1) .. controls (455.43,237.98) and (480.5,244.5) .. (480.5,252.25) .. controls (480.5,261.48) and (445.02,268.96) .. (401.25,268.96) .. controls (357.48,268.96) and (322,261.48) .. (322,252.25) .. controls (322,244.58) and (346.58,238.11) .. (380.06,236.16) -- (401.25,252.25) -- cycle ; \draw   (421.44,236.1) .. controls (455.43,237.98) and (480.5,244.5) .. (480.5,252.25) .. controls (480.5,261.48) and (445.02,268.96) .. (401.25,268.96) .. controls (357.48,268.96) and (322,261.48) .. (322,252.25) .. controls (322,244.58) and (346.58,238.11) .. (380.06,236.16) ;  
\draw    (400.94,252.42) .. controls (434.22,227.36) and (447.69,205.65) .. (454.03,170.58) .. controls (460.37,135.5) and (441.35,113.79) .. (401.73,112.95) .. controls (362.1,112.12) and (342.29,135.5) .. (347.05,169.74) .. controls (351.8,203.98) and (383.5,241.56) .. (400.94,252.42) -- cycle ;
\draw  [draw opacity=0][dash pattern={on 4.5pt off 4.5pt}] (347.39,169.44) .. controls (348.96,160) and (372.29,152.51) .. (400.83,152.51) .. controls (430.34,152.51) and (454.27,160.52) .. (454.35,170.41) -- (400.83,170.46) -- cycle ; \draw  [dash pattern={on 4.5pt off 4.5pt}] (347.39,169.44) .. controls (348.96,160) and (372.29,152.51) .. (400.83,152.51) .. controls (430.34,152.51) and (454.27,160.52) .. (454.35,170.41) ;  
\draw  [draw opacity=0] (454.06,170.48) .. controls (451.29,178.9) and (428.51,185.46) .. (400.82,185.46) .. controls (372.19,185.46) and (348.8,178.44) .. (347.37,169.61) -- (400.82,168.76) -- cycle ; \draw   (454.06,170.48) .. controls (451.29,178.9) and (428.51,185.46) .. (400.82,185.46) .. controls (372.19,185.46) and (348.8,178.44) .. (347.37,169.61) ;  
\draw  [fill={rgb, 255:red, 224; green, 16; blue, 16 }  ,fill opacity=1 ] (442.46,262.27) -- (443.62,264.76) -- (446.23,265.16) -- (444.34,267.1) -- (444.79,269.83) -- (442.46,268.54) -- (440.13,269.83) -- (440.58,267.1) -- (438.69,265.16) -- (441.3,264.76) -- cycle ;
\draw  [fill={rgb, 255:red, 7; green, 0; blue, 0 }  ,fill opacity=1 ] (398.48,252.25) .. controls (398.48,250.64) and (399.72,249.33) .. (401.25,249.33) .. controls (402.78,249.33) and (404.02,250.64) .. (404.02,252.25) .. controls (404.02,253.87) and (402.78,255.18) .. (401.25,255.18) .. controls (399.72,255.18) and (398.48,253.87) .. (398.48,252.25) -- cycle ;
\draw  [color={rgb, 255:red, 18; green, 6; blue, 247 }  ,draw opacity=1 ] (398.87,183) -- (403.95,188.55)(404.11,183.04) -- (398.71,188.51) ;
\draw  [color={rgb, 255:red, 18; green, 6; blue, 247 }  ,draw opacity=1 ] (344.19,167.13) -- (349.26,172.69)(349.42,167.17) -- (344.03,172.65) ;
\draw  [color={rgb, 255:red, 18; green, 6; blue, 247 }  ,draw opacity=1 ] (451.97,167.97) -- (457.04,173.52)(457.2,168.01) -- (451.81,173.48) ;
\draw  [color={rgb, 255:red, 18; green, 6; blue, 247 }  ,draw opacity=1 ] (399.66,150.43) -- (404.74,155.99)(404.9,150.47) -- (399.5,155.94) ;
\draw  [fill={rgb, 255:red, 224; green, 16; blue, 16 }  ,fill opacity=1 ] (402.04,108.62) -- (403.21,111.1) -- (405.81,111.5) -- (403.93,113.44) -- (404.37,116.17) -- (402.04,114.88) -- (399.71,116.17) -- (400.16,113.44) -- (398.27,111.5) -- (400.88,111.1) -- cycle ;
\draw   (34,199) .. controls (34,162) and (64,132) .. (101,132) .. controls (138,132) and (168,162) .. (168,199) .. controls (168,236) and (138,266) .. (101,266) .. controls (64,266) and (34,236) .. (34,199) -- cycle ;
\draw  [fill={rgb, 255:red, 224; green, 16; blue, 16 }  ,fill opacity=1 ] (141.04,249.62) -- (142.21,252.1) -- (144.81,252.5) -- (142.93,254.44) -- (143.37,257.17) -- (141.04,255.88) -- (138.71,257.17) -- (139.16,254.44) -- (137.27,252.5) -- (139.88,252.1) -- cycle ;
\draw  [fill={rgb, 255:red, 224; green, 16; blue, 16 }  ,fill opacity=1 ] (99.84,196.51) -- (101,199) -- (103.6,199.4) -- (101.72,201.33) -- (102.16,204.07) -- (99.84,202.78) -- (97.51,204.07) -- (97.95,201.33) -- (96.07,199.4) -- (98.67,199) -- cycle ;
\draw  [color={rgb, 255:red, 18; green, 6; blue, 247 }  ,draw opacity=1 ] (148.87,199) -- (153.95,204.55)(154.11,199.04) -- (148.71,204.51) ;
\draw  [color={rgb, 255:red, 18; green, 6; blue, 247 }  ,draw opacity=1 ] (97.87,149) -- (102.95,154.55)(103.11,149.04) -- (97.71,154.51) ;
\draw  [color={rgb, 255:red, 18; green, 6; blue, 247 }  ,draw opacity=1 ] (97.87,248) -- (102.95,253.55)(103.11,248.04) -- (97.71,253.51) ;
\draw  [color={rgb, 255:red, 18; green, 6; blue, 247 }  ,draw opacity=1 ] (47.87,198) -- (52.95,203.55)(53.11,198.04) -- (47.71,203.51) ;
\draw    (105.5,201.23) -- (146.5,201.23) ;
\draw    (105.5,201.23) -- (109.02,206.73) ;
\draw    (146.5,201.23) -- (142.98,207) ;
\draw    (105.5,201.23) -- (108.81,196) ;
\draw    (146.5,201.23) -- (143.19,196) ;
\draw    (202,201) -- (298.5,201) ;
\draw [shift={(300.5,201)}, rotate = 180] [color={rgb, 255:red, 0; green, 0; blue, 0 }  ][line width=0.75]    (10.93,-3.29) .. controls (6.95,-1.4) and (3.31,-0.3) .. (0,0) .. controls (3.31,0.3) and (6.95,1.4) .. (10.93,3.29)   ;

\draw (446.67,266.26) node [anchor=north west][inner sep=0.75pt]   [align=left] {$\chi$};
\draw (396.15,93.02) node [anchor=north west][inner sep=0.75pt]   [align=left] {$0$};
\draw (459.55,161.5) node [anchor=north west][inner sep=0.75pt]   [align=left] {$1$};
\draw (404.86,135.61) node [anchor=north west][inner sep=0.75pt]   [align=left] {$\zeta$};
\draw (134,262) node [anchor=north west][inner sep=0.75pt]   [align=left] {$\chi$};
\draw (122,184) node [anchor=north west][inner sep=0.75pt]   [align=left] {$t$};
\draw (228,173) node [anchor=north west][inner sep=0.75pt]   [align=left] {$t\rightarrow 0$};

\end{tikzpicture}
    \caption{domain-stable components of the limit, $r=4$.}
    \label{Gromov limit figure}
\end{figure}
\begin{lemma}\label{spherical bubble}
Suppose that $r\leq j_Z$. Let $(u_k)\in \mathscr{T}_0^{\chi}(\alpha)$ be a sequence such that
\begin{equation*}
    \Phi_{\alpha}(u_k) = t_k[\mathbf{z}_1] = (t_k\zeta,\dots,t_k\zeta^r)\in \text{Sym}^{\mathbf{v}}(\mathbb{D}) ,\quad t_k\rightarrow 0.
\end{equation*}
   If the sequence $(u_k)$ is not $C^1$-bounded, then it converges (after passing to a subsequence) to the constant disc $\mathscr{M}_{\chi}$ attached at $0\in\mathbb{D}$ to a spherical bubble $u_{\infty}$ from the moduli space
\begin{equation}\label{moduli of spherical bubbles}
    \mathscr{M}^{s}_{\mathbf{z}_1}(Z,D,p,\alpha) = \{u\in \mathscr{M}^{s}_{\mathbf{z}_1}(Z,D,p) \ | \ [u]=\alpha\}.
\end{equation}
\end{lemma}
\begin{proof}
The convergence stated above is in the Gromov topology. The marked domains $(\mathbb{D},0,\chi,t_k\zeta,\dots,t_k\zeta^r)$ associated with $(u_k)$ converge (in the Deligne-Mumford space of discs with one boundary marked point and $r+1$ interior marked points) to a disc $\mathbb{D}$ with one boundary marked point $\chi\in\partial\mathbb{D}$ that is attached to a sphere $(\mathbb{P}^1,0,\zeta,\dots,\zeta^r)$. The attachment identifies $0\in\mathbb{D}$ with $\infty\in\mathbb{P}^1$, see Figure \ref{Gromov limit figure}.
By Lemma \ref{no disc bubbles}, the sequence does not exhibit disc bubbles. Furthermore, the assumption $r\leq j_Z$ ensures that if a spherical bubble arises, it will be the only non-constant component. In particular, the fundamental component of the limit is constant. This implies that bubbling occurs at $z=0 \in\mathbb{D}$ because $u_k(t_k\zeta^i)\in D$, $D$ is disjoint from $L$, and $t_k$ converges to $0$.
\end{proof}

We now study compactness of the moduli space $\{(u,t)\ |\ \Phi_{\alpha}(u)=t[z_1]\}$ as $t\rightarrow 1$. Let $D_i$ be a smooth component of the simple normal crossings divisor $D$. For each root of unity $\zeta^k\in\partial \mathbb{D}$ and relative homology class $\beta\in H_2(Z,L)$, consider the moduli space of pseudo-holomorphic discs
\begin{equation*}
    \mathscr{M}^{D_i}_{\zeta^k}(L,\beta) = \{ v :(\mathbb{D},\partial\mathbb{D},0)\rightarrow (Z,L,D_i) |\ \overline{\partial}_J v = 0, v(-\zeta^k)=p,[v]=\beta \}.
\end{equation*}
This moduli space is essentially the same as $\mathscr{T}^{\chi}_{1}(\beta)$ defined in (\ref{tangency with homology class}), except that the $\overline{\partial}$-equation uses a constant $J$ as in (\ref{discs no constraint}).

\begin{lemma}\label{count agrees with classical count}
 Suppose that
 \begin{equation}\label{constraint for small discs}
     \mu_L(\beta)=2 \quad\text{and}\quad \beta\cdot D_i = 1.
 \end{equation}
 Then $\mathscr{M}^{D_i}_{\zeta^k}(L,\beta)$ is a closed oriented manifold of dimension $0$. The count of elements in $\mathscr{M}^{D_i}_{\zeta^k}(L,\beta)$ is $m_{0,\beta}(L)$.
\end{lemma}
\begin{proof}
We interpolate between the two counts using the moduli space
\begin{gather*}
  \mathscr{M}_{D_i}(L,\beta) =\{ v :\disk\rightarrow (Z,L) |\ \overline{\partial}_J v = 0, v(0)\in D_i, [v]=\beta\}.
\end{gather*}
Again, by the structure theorem of Lazzarini (see \cite[Theorem A]{Lazzarini-simplediscs}), all discs in this moduli space are \`a-priori somewhere-injective. Thus, for a generic choice of $J$, both moduli spaces must be regular. The point $p\in L$ is (by assumption) transverse to the map
\begin{equation*}
    ev_{-\zeta^k}: \mathscr{M}_{D_i}(L,\beta)\rightarrow L; \quad \mathscr{M}^{D_i}_{\zeta^i}(L,\beta) = ev_{-\zeta^k}^{-1}(p).
\end{equation*}
Using Lemma \ref{compactness for tangency discs} and the remark thereafter, $ev_{-\zeta^ik}$ is a pseudo-cycle whose degree is
\begin{equation*}
    \deg(ev_{-\zeta^k})=\#\mathscr{M}^{D_i}_{\zeta^k}(L,\beta).
\end{equation*}
Recall that $m_{0,\beta}(L)$ is the degree of the pseudo-cycle
\begin{equation*}
    ev:\mathscr{M}_{0,1}(L,\beta)\defeq \mathscr{M}(L,\beta)\times\partial \mathbb{D}/\text{Aut}(\mathbb{D})\rightarrow L.
\end{equation*}
But since $\beta\cdot D_i = 1$, the tautological map
\begin{equation*}
\iota : \mathscr{M}_{D_j}(L,\beta)\rightarrow \mathscr{M}_{0,1}(L,\beta),\quad v\mapsto [v,-\zeta^k]
\end{equation*}
is a diffeomorphism. The lemma now follows from the identity $ev_{-\zeta^k}=ev\circ\iota$.
\end{proof}

\begin{lemma}\label{disc bubbles at t=1}
Let $(u_k)\in \mathscr{T}_0^{\chi}(\alpha)$ be a sequence such that
\begin{equation*}
    \Phi_{\alpha}(u_k) = t_k[\mathbf{z}_1] = (t_k\zeta,\dots,t_k\zeta^r),\quad t_k\rightarrow 1.
\end{equation*}
   Then, $(u_k)$ converges (after passing to a subsequence) to a nodal disc
   \begin{equation}\label{small discs I}
    u_{\infty}\in \mathscr{M}_{\chi}\bigtimes_{i=1}^N\bigtimes_{j=1}^{r_i} \mathscr{M}^{D_i}_{(\boldsymbol{\zeta}_i)_j}(L,\beta^i_j), \quad \alpha = \sum_{i,j}\beta^i_j, \quad \mu_L(\beta^i_j)=2,\quad \beta_i^j\cdot D_i=1
   \end{equation}
   where $(\boldsymbol{\zeta}_i)_j$ is the $j^{\text{th}}$ component of $\boldsymbol{\zeta}_i$.
   The fundamental component of the limit $u_{\infty}$ is a constant disc $\mathscr{M}_{\chi}$. It is attached to a collection of discs $v^i_j\in \mathscr{M}^{D_i}_{(\boldsymbol{\zeta}_i)_j}(L,\beta^i_j)$ at the roots of unity $(\boldsymbol{\zeta}_i)_j\in\partial\mathbb{D}$.
\end{lemma}
\begin{proof}
Identical to Lemma \ref{spherical bubble}, see also Figure \ref{gluing map figure}.
\end{proof}

In order to simplify notation, we can forget about the decomposition of $D$ into its irreducible components $D_i$. The limiting element $u_{\infty}$ from the previous lemma can then be thought of as
\begin{equation}\label{small discs II}
    u_{\infty}\in \mathscr{M}_{\chi}\bigtimes_{k=1}^r\mathscr{M}_{\zeta^k}^D(L,\beta_k),
\end{equation}
where the classes $\beta_k$ satisfy the conditions
\begin{equation}\label{small discs}
    \alpha = \beta_1+\dots+\beta_r,\quad \mu_L(\beta_k)=2,\quad \beta_k\cdot D =1.
\end{equation}
The elements of the moduli space (\ref{small discs II}) are essentially the same as those of (\ref{small discs I}), but over-counted with a factor of
\begin{equation*}
    \frac{r!}{r_1!\dotsi r_N!}.
\end{equation*}
\subsection{Counting} 
The final ingredients we need are gluing and orientations. We explain how some of the theory appearing in the literature applies to our setup.
\subsubsection{Gluing} We rely on the work of Fukaya-Oh-Ohta-Ono in \cite{FO+gluing} for gluing analysis, particularly Theorems 3.13 and 8.16.
Let $\beta_1,\dots,\beta_r\in H_2(Z,L)$ be an ordered collection of relative homology classes, satisfying the constrains in (\ref{small discs}). Set
\begin{equation*}
    \alpha = \beta_1+\dots+\beta_r.
\end{equation*}
Then, there is an open embedding
\begin{equation}\label{gluing-map}
   \Psi=(G,t) :(R,\infty) \times \mathscr{M}_{\chi}\bigtimes_{k=1}^r \mathscr{M}^D_{\zeta^k}(L,\beta_k)  \rightarrow \mathscr{T}_0^{\chi}(\alpha)\times (0,1)
\end{equation}
whenever the gluing length $R$ is sufficiently large. The first component $G$ of this map is the result of gluing a collection of discs $v_i\in \mathscr{M}^D_{\zeta^k}(L,\beta_k)$ to the constant disc $\mathscr{M}_{\chi}$ at the roots of unity $\zeta^i$ in the domain of $\mathscr{M}_{\chi}$ and using a gluing length $\rho\in (R,\infty)$, this is the map described in Theorem 3.13 of \cite{FO+gluing}.

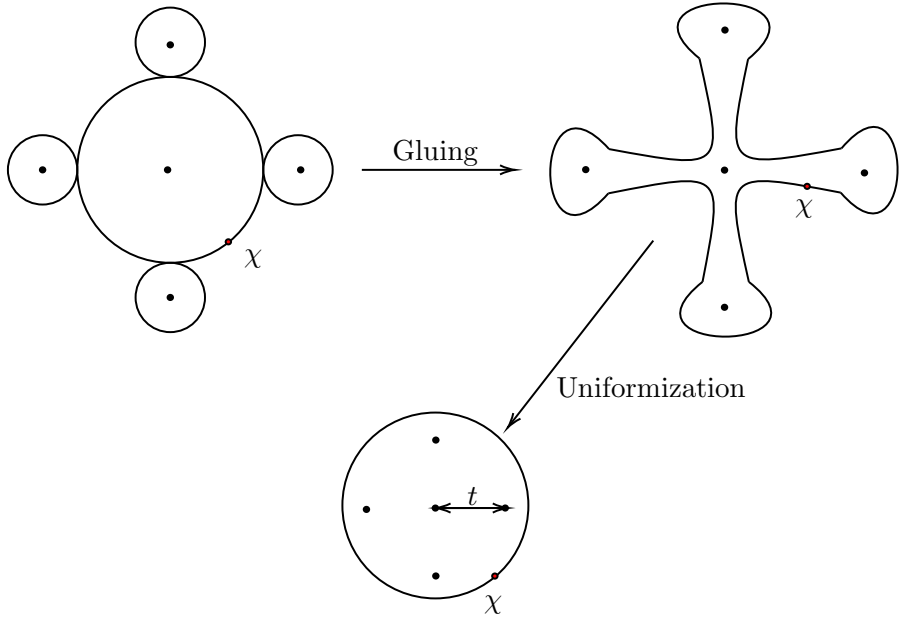
\begin{figure}
    \centering
    \begin{tikzpicture}[x=0.75pt,y=0.75pt,yscale=-.7,xscale=.7]
\draw    (458.16,136.03) .. controls (401.98,73.28) and (401.43,232.29) .. (457.61,167.94) ;
\draw    (625.6,135.5) .. controls (679.56,71.68) and (679.56,230.17) .. (625.04,167.41) ;
\draw    (523.25,70.62) .. controls (457.61,17.97) and (623.38,17.44) .. (558.29,71.15) ;
\draw    (525.47,231.76) .. controls (458.72,284.95) and (625.6,283.88) .. (558.85,231.76) ;
\draw    (458.16,136.03) .. controls (542.16,151.46) and (541.05,151.46) .. (523.25,70.62) ;
\draw    (457.61,167.94) .. controls (541.6,151.46) and (541.05,150.93) .. (525.47,231.76) ;
\draw    (558.29,71.15) .. controls (541.6,150.93) and (542.16,150.93) .. (625.6,135.5) ;
\draw    (558.85,231.76) .. controls (541.05,151.46) and (541.05,151.46) .. (625.04,167.41) ;
\draw  [fill={rgb, 255:red, 0; green, 0; blue, 0 }  ,fill opacity=1 ] (539.38,151.19) .. controls (539.38,150.16) and (540.25,149.33) .. (541.33,149.33) .. controls (542.4,149.33) and (543.27,150.16) .. (543.27,151.19) .. controls (543.27,152.22) and (542.4,153.05) .. (541.33,153.05) .. controls (540.25,153.05) and (539.38,152.22) .. (539.38,151.19) -- cycle ;
\draw  [fill={rgb, 255:red, 0; green, 0; blue, 0 }  ,fill opacity=1 ] (539.38,250.09) .. controls (539.38,249.06) and (540.25,248.23) .. (541.33,248.23) .. controls (542.4,248.23) and (543.27,249.06) .. (543.27,250.09) .. controls (543.27,251.12) and (542.4,251.95) .. (541.33,251.95) .. controls (540.25,251.95) and (539.38,251.12) .. (539.38,250.09) -- cycle ;
\draw  [fill={rgb, 255:red, 0; green, 0; blue, 0 }  ,fill opacity=1 ] (439.38,150.83) .. controls (439.38,149.8) and (440.25,148.97) .. (441.33,148.97) .. controls (442.4,148.97) and (443.27,149.8) .. (443.27,150.83) .. controls (443.27,151.86) and (442.4,152.69) .. (441.33,152.69) .. controls (440.25,152.69) and (439.38,151.86) .. (439.38,150.83) -- cycle ;
\draw  [fill={rgb, 255:red, 0; green, 0; blue, 0 }  ,fill opacity=1 ] (640.19,153.19) .. controls (640.19,152.16) and (641.06,151.33) .. (642.14,151.33) .. controls (643.21,151.33) and (644.08,152.16) .. (644.08,153.19) .. controls (644.08,154.22) and (643.21,155.05) .. (642.14,155.05) .. controls (641.06,155.05) and (640.19,154.22) .. (640.19,153.19) -- cycle ;
\draw  [fill={rgb, 255:red, 0; green, 0; blue, 0 }  ,fill opacity=1 ] (539.68,50.19) .. controls (539.68,49.16) and (540.55,48.33) .. (541.63,48.33) .. controls (542.7,48.33) and (543.58,49.16) .. (543.58,50.19) .. controls (543.58,51.22) and (542.7,52.05) .. (541.63,52.05) .. controls (540.55,52.05) and (539.68,51.22) .. (539.68,50.19) -- cycle ;
\draw  [fill={rgb, 255:red, 245; green, 9; blue, 9 }  ,fill opacity=1 ] (598.9,162.89) .. controls (598.9,161.86) and (599.77,161.03) .. (600.85,161.03) .. controls (601.92,161.03) and (602.79,161.86) .. (602.79,162.89) .. controls (602.79,163.92) and (601.92,164.75) .. (600.85,164.75) .. controls (599.77,164.75) and (598.9,163.92) .. (598.9,162.89) -- cycle ;
\draw   (75,151) .. controls (75,114) and (105,84) .. (142,84) .. controls (179,84) and (209,114) .. (209,151) .. controls (209,188) and (179,218) .. (142,218) .. controls (105,218) and (75,188) .. (75,151) -- cycle ;
\draw   (209,151) .. controls (209,137.19) and (220.19,126) .. (234,126) .. controls (247.81,126) and (259,137.19) .. (259,151) .. controls (259,164.81) and (247.81,176) .. (234,176) .. controls (220.19,176) and (209,164.81) .. (209,151) -- cycle ;
\draw   (25,151) .. controls (25,137.19) and (36.19,126) .. (50,126) .. controls (63.81,126) and (75,137.19) .. (75,151) .. controls (75,164.81) and (63.81,176) .. (50,176) .. controls (36.19,176) and (25,164.81) .. (25,151) -- cycle ;
\draw   (117,243) .. controls (117,229.19) and (128.19,218) .. (142,218) .. controls (155.81,218) and (167,229.19) .. (167,243) .. controls (167,256.81) and (155.81,268) .. (142,268) .. controls (128.19,268) and (117,256.81) .. (117,243) -- cycle ;
\draw   (117,59) .. controls (117,45.19) and (128.19,34) .. (142,34) .. controls (155.81,34) and (167,45.19) .. (167,59) .. controls (167,72.81) and (155.81,84) .. (142,84) .. controls (128.19,84) and (117,72.81) .. (117,59) -- cycle ;
\draw   (266,393) .. controls (266,356) and (296,326) .. (333,326) .. controls (370,326) and (400,356) .. (400,393) .. controls (400,430) and (370,460) .. (333,460) .. controls (296,460) and (266,430) .. (266,393) -- cycle ;
\draw  [fill={rgb, 255:red, 245; green, 9; blue, 9 }  ,fill opacity=1 ] (373.9,443.89) .. controls (373.9,442.86) and (374.77,442.03) .. (375.85,442.03) .. controls (376.92,442.03) and (377.79,442.86) .. (377.79,443.89) .. controls (377.79,444.92) and (376.92,445.75) .. (375.85,445.75) .. controls (374.77,445.75) and (373.9,444.92) .. (373.9,443.89) -- cycle ;
\draw  [fill={rgb, 255:red, 245; green, 9; blue, 9 }  ,fill opacity=1 ] (181.9,202.89) .. controls (181.9,201.86) and (182.77,201.03) .. (183.85,201.03) .. controls (184.92,201.03) and (185.79,201.86) .. (185.79,202.89) .. controls (185.79,203.92) and (184.92,204.75) .. (183.85,204.75) .. controls (182.77,204.75) and (181.9,203.92) .. (181.9,202.89) -- cycle ;
\draw  [fill={rgb, 255:red, 0; green, 0; blue, 0 }  ,fill opacity=1 ] (140.05,60.86) .. controls (140.05,59.83) and (140.92,59) .. (142,59) .. controls (143.08,59) and (143.95,59.83) .. (143.95,60.86) .. controls (143.95,61.89) and (143.08,62.72) .. (142,62.72) .. controls (140.92,62.72) and (140.05,61.89) .. (140.05,60.86) -- cycle ;
\draw  [fill={rgb, 255:red, 0; green, 0; blue, 0 }  ,fill opacity=1 ] (48.05,151) .. controls (48.05,149.97) and (48.92,149.14) .. (50,149.14) .. controls (51.08,149.14) and (51.95,149.97) .. (51.95,151) .. controls (51.95,152.03) and (51.08,152.86) .. (50,152.86) .. controls (48.92,152.86) and (48.05,152.03) .. (48.05,151) -- cycle ;
\draw  [fill={rgb, 255:red, 0; green, 0; blue, 0 }  ,fill opacity=1 ] (140.05,243) .. controls (140.05,241.97) and (140.92,241.14) .. (142,241.14) .. controls (143.08,241.14) and (143.95,241.97) .. (143.95,243) .. controls (143.95,244.03) and (143.08,244.86) .. (142,244.86) .. controls (140.92,244.86) and (140.05,244.03) .. (140.05,243) -- cycle ;
\draw  [fill={rgb, 255:red, 0; green, 0; blue, 0 }  ,fill opacity=1 ] (234,151) .. controls (234,149.97) and (234.87,149.14) .. (235.95,149.14) .. controls (237.02,149.14) and (237.89,149.97) .. (237.89,151) .. controls (237.89,152.03) and (237.02,152.86) .. (235.95,152.86) .. controls (234.87,152.86) and (234,152.03) .. (234,151) -- cycle ;
\draw  [fill={rgb, 255:red, 0; green, 0; blue, 0 }  ,fill opacity=1 ] (331.05,394.86) .. controls (331.05,393.83) and (331.92,393) .. (333,393) .. controls (334.08,393) and (334.95,393.83) .. (334.95,394.86) .. controls (334.95,395.89) and (334.08,396.72) .. (333,396.72) .. controls (331.92,396.72) and (331.05,395.89) .. (331.05,394.86) -- cycle ;
\draw    (280,151) -- (388.5,151) ;
\draw [shift={(390.5,151)}, rotate = 180] [color={rgb, 255:red, 0; green, 0; blue, 0 }  ][line width=0.75]    (10.93,-3.29) .. controls (6.95,-1.4) and (3.31,-0.3) .. (0,0) .. controls (3.31,0.3) and (6.95,1.4) .. (10.93,3.29)   ;
\draw    (490,202) -- (386.23,335.42) ;
\draw [shift={(385,337)}, rotate = 307.87] [color={rgb, 255:red, 0; green, 0; blue, 0 }  ][line width=0.75]    (10.93,-3.29) .. controls (6.95,-1.4) and (3.31,-0.3) .. (0,0) .. controls (3.31,0.3) and (6.95,1.4) .. (10.93,3.29)   ;
\draw  [fill={rgb, 255:red, 0; green, 0; blue, 0 }  ,fill opacity=1 ] (281.38,395.83) .. controls (281.38,394.8) and (282.25,393.97) .. (283.33,393.97) .. controls (284.4,393.97) and (285.27,394.8) .. (285.27,395.83) .. controls (285.27,396.86) and (284.4,397.69) .. (283.33,397.69) .. controls (282.25,397.69) and (281.38,396.86) .. (281.38,395.83) -- cycle ;
\draw  [fill={rgb, 255:red, 0; green, 0; blue, 0 }  ,fill opacity=1 ] (381.38,394.83) .. controls (381.38,393.8) and (382.25,392.97) .. (383.33,392.97) .. controls (384.4,392.97) and (385.27,393.8) .. (385.27,394.83) .. controls (385.27,395.86) and (384.4,396.69) .. (383.33,396.69) .. controls (382.25,396.69) and (381.38,395.86) .. (381.38,394.83) -- cycle ;
\draw  [fill={rgb, 255:red, 0; green, 0; blue, 0 }  ,fill opacity=1 ] (331.38,345.83) .. controls (331.38,344.8) and (332.25,343.97) .. (333.33,343.97) .. controls (334.4,343.97) and (335.27,344.8) .. (335.27,345.83) .. controls (335.27,346.86) and (334.4,347.69) .. (333.33,347.69) .. controls (332.25,347.69) and (331.38,346.86) .. (331.38,345.83) -- cycle ;
\draw  [fill={rgb, 255:red, 0; green, 0; blue, 0 }  ,fill opacity=1 ] (331.38,443.83) .. controls (331.38,442.8) and (332.25,441.97) .. (333.33,441.97) .. controls (334.4,441.97) and (335.27,442.8) .. (335.27,443.83) .. controls (335.27,444.86) and (334.4,445.69) .. (333.33,445.69) .. controls (332.25,445.69) and (331.38,444.86) .. (331.38,443.83) -- cycle ;
\draw    (331.05,394.86) -- (383.27,394.83) ;
\draw [shift={(385.27,394.83)}, rotate = 179.96] [color={rgb, 255:red, 0; green, 0; blue, 0 }  ][line width=0.75]    (10.93,-3.29) .. controls (6.95,-1.4) and (3.31,-0.3) .. (0,0) .. controls (3.31,0.3) and (6.95,1.4) .. (10.93,3.29)   ;
\draw    (381.38,394.83) -- (335,394.86) ;
\draw [shift={(333,394.86)}, rotate = 359.96] [color={rgb, 255:red, 0; green, 0; blue, 0 }  ][line width=0.75]    (10.93,-3.29) .. controls (6.95,-1.4) and (3.31,-0.3) .. (0,0) .. controls (3.31,0.3) and (6.95,1.4) .. (10.93,3.29)   ;
\draw  [fill={rgb, 255:red, 0; green, 0; blue, 0 }  ,fill opacity=1 ] (138.11,151) .. controls (138.11,149.97) and (138.98,149.14) .. (140.05,149.14) .. controls (141.13,149.14) and (142,149.97) .. (142,151) .. controls (142,152.03) and (141.13,152.86) .. (140.05,152.86) .. controls (138.98,152.86) and (138.11,152.03) .. (138.11,151) -- cycle ;

\draw (193,206) node [anchor=north west][inner sep=0.75pt]   [align=left] {$\chi$};
\draw (590,170) node [anchor=north west][inner sep=0.75pt]   [align=left] {$\chi$};
\draw (366,456) node [anchor=north west][inner sep=0.75pt]   [align=left] {$\chi$};
\draw (300,127) node [anchor=north west][inner sep=0.75pt]   [align=left] {Gluing};
\draw (418,298) node [anchor=north west][inner sep=0.75pt]   [align=left] {Uniformization};
\draw (354,378) node [anchor=north west][inner sep=0.75pt]   [align=left] {$t$};

\end{tikzpicture}
    \caption{Gluing map for $r=4$; length of $t$ represents the second component of (\ref{gluing-map}).}
    \label{gluing map figure}
\end{figure}

 When the $r+1$ domain discs are glued, the resulting Riemann surface is biholomorphic to a disc. The choice of a biholomorphism is determined by the interior marked point $0$ and the boundary marked point $\chi$. We call this choice a \emph{uniformization}. The second component of the map (\ref{gluing-map}) tracks how far the new extra marked points are from $0$ in the \emph{uniformization} of the glued (domain) disc, see Figure \ref{gluing map figure}.

Consider the evaluation map
\begin{align*}
  ev:   \mathscr{T}_0^{\chi}\times (0,1) &\rightarrow Z^r\\
        (u,t)&\mapsto (u(t\zeta ),\dots, u(t\zeta^r )).
\end{align*}

\begin{lemma}\label{tranversality near boundary}
When the gluing length $\rho_+\in (R,\infty)$ is sufficiently large, the composition $ev\circ\Psi_{|\rho=\rho_+}$ is transverse to $D^r\subseteq Z^r$.
\end{lemma}

\begin{proof} By Gromov compactness, a sequence $(\rho_l,(v_k)_l) \in (ev\circ\Psi)^{-1}(D^r)$, with $\rho_k\rightarrow \infty$ must have a converging sub-sequence of $(v_k)_l$. Moreover, at $\rho=\infty$, the composition $ev\circ\Psi$ is the product map
\begin{align*}
    \bigtimes_{k=1}^r \mathscr{M}^D_{\zeta^k}(L,\beta_k) &\rightarrow Z^r\\
    (v_k) &\mapsto (v_1(0),\dots, v_r(0)),
\end{align*}
which is transverse to $D^r$.
\end{proof}
\begin{remark}
    Some of the literature uses a gluing \emph{parameter} $\delta = e^{-\rho}\in (0,\epsilon)$ instead of a gluing length $\rho\in (R,\infty)$. One advantage is that the nodal curves at the boundary of the moduli space correspond to $\delta = 0$, as opposed to $\rho = \infty$.
\end{remark}
\subsubsection{Orientations} We now briefly explain how the various moduli spaces we've been studying are oriented. We use the notation $\lambda^{\tp}(V)$ for the top exterior power of a finite dimensional vector space $V$, and $\det(D)$ for the determinant line of a Fredholm operator $D$ between Banach spaces,
\begin{equation*}
    \det(D) \defeq \lambda^{\tp}(\coker(D))^{\vee} \otimes \lambda^{\tp}(\ker(D)).
\end{equation*}

The orientation problem for the moduli space $\mathscr{T}_0(\alpha)$ amounts to choosing trivializations of the determinant lines $\det(D_u)$ depending continuously on $u\in \mathscr{T}_0(\alpha)$, where $D_u$ is the Cauchy-Riemann operator obtained by linearizing equation \eqref{pseudo-holomorphic-equation}. Let us denote the line bundle formed using the lines $\det(D_u)$ by $\underline{\det}_{\alpha}\rightarrow \mathscr{T}_0(\alpha)$. Let $u_{\theta}: S^1\rightarrow \mathscr{T}_0(\alpha)$ be a loop of pseudo-holomoprhic discs. Then, Lemma 11.7 of \cite[(11e)]{PL} shows that because $L$ is orientable, we have
\begin{equation*}
    \langle w_1(\underline{\det}_{\alpha}),[u_{\theta}]\rangle = \langle w_2(L), [T]\rangle,
\end{equation*}
where $T: S^1\times S^1 \rightarrow L$, given by $(\theta_1,\theta_2)\mapsto u_{\theta_1}(\theta_2)$, is the torus swept by the boundaries of the pseudo-holomoprhic discs $u_{\theta}$.

Moreover, the choice of a spin structure on $L$ canonically determines a trivialization of $\underline{\det}_{\alpha}$ across all classes $\alpha$ and these trivializations are consistent with gluing. This is essentially the content of Lemma 11.12 of \cite[(11h)]{PL} and the gluing formulae (11.11) and (11.12) of \cite[(11c)]{PL}. A more detailed description of this orientation process can be found in the proof of \cite[Theorem 8.1.1]{FO3-book2}, and compatibility with gluing is proved in \cite[Lemma 8.3.5]{FO3-book2}. From now on, we assume that $L$ is equipped with a spin structure and we use the induced trivializations $\det(D_u)\cong \mathbb{R}$ on determinant lines.

Using Lemma \ref{transversality} and the usual orientation of $\mathbb{C}$, we inductively obtain orientations
\begin{equation}\label{orientation}
    \lambda^{\tp}T_u\mathscr{T}_{\mathbf{v}}(\alpha) \cong \mathbb{R}
\end{equation}
for the moduli spaces $\mathscr{T}_{\mathbf{v}}(\alpha)$ for various tangency vectors $\mathbf{v}$. Finally, the moduli spaces $\mathscr{T}^{\chi}_{\mathbf{v}}(\alpha)$ are oriented using the short exact sequence
\begin{equation*}
    0\rightarrow T_u\mathscr{T}^{\chi}_{\mathbf{v}}(\alpha)\rightarrow T_u\mathscr{T}_{\mathbf{v}}(\alpha) \xrightarrow[]{d_uev_{\chi}} T_{u(\chi)}L\rightarrow 0,
\end{equation*}
together with the orientation given to $L$. 

The previous discussion covers the orientation problem for the moduli spaces $\mathscr{M}(L,\beta),\mathscr{M}_D(L,\beta)$ and $\mathscr{M}^D_{\zeta^i}(L,\beta)$ from Lemma \ref{count agrees with classical count}. The moduli space 
\begin{equation*}
    \mathscr{M}_{0,1}(L,\beta)\defeq \mathscr{M}(L,\beta)\times\partial \mathbb{D}/\text{Aut}(\mathbb{D})
\end{equation*}
is oriented using the isomorphism
\begin{equation*}
    T_{[(u,\xi)]}\mathscr{M}_{0,1}(L,\beta)\oplus T_{\id}\text{Aut}(\mathbb{D}) \cong T_u\mathscr{M}(L,\beta)\oplus T_{\xi}\mathbb{\partial D}.
\end{equation*}
This orientation does not depend on the choice of lift $(u,\xi)$. Indeed, two such lifts are related by the action of $\text{Aut}(\mathbb{D})$ on $\mathscr{M}(L,\beta)\times\partial \mathbb{D}$, which preserves the orientation described in \eqref{orientation}.

\subsubsection{Homotopy} Let $\mathbf{z}_1=(\zeta,\dots\zeta^r)$ be the collection of $r^{\text{th}}$-roots of unity. By Lemmas \ref{tranversality near boundary}, \ref{transversality by evaluation}, and \ref{transversality near 0}, the points
\begin{equation}
    [\mathbf{z}_t] = t\times [\mathbf{z}_1]
\end{equation}
are regular values of the map $\Phi_{\alpha}: \mathscr{T}_0^{\chi}(\alpha)\rightarrow \Sym^{\mathbf{v}}(\mathbb{D})$ from (\ref{Phi map}), when $t$ is close to $0$ or to $1$. Let $\gamma:[0,1)\rightarrow \Sym^{\mathbf{v}}(\mathbb{D})$ be an embedded path which is disjoint from the big diagonal of $\Sym^{\mathbf{v}}(\mathbb{D})$ (except at $t=0$), and such that
\begin{equation*}
    \gamma(t) = [\mathbf{z}_t] \ \ \text{for}\ t\in[0,\epsilon)\cup(1-\epsilon,1).
\end{equation*}
Then, assuming $\gamma$ is generic, the pre-image $\Phi_{\alpha}^{-1}(\gamma)$ is a smooth oriented $1$-dimensional manifold with boundary.  It is oriented using the isomorphism
\begin{equation}\label{orientation of homotopy moduli}
   \mathbb{R}\langle\partial_t\rangle\otimes \lambda^{\tp} T_{(u,t)}\Phi_{\alpha}^{-1}(\gamma) \cong \det(D_u)
\end{equation}
which in turn is obtained from the short exact sequence
\begin{equation*}
   0\rightarrow T_{(u,t)}\Phi_{\alpha}^{-1}(\gamma) \hookrightarrow T_u \mathscr{T}_0^{\chi}(\alpha)\rightarrow T_{\gamma(t)} \Sym^{\mathbf{v}}(\mathbb{D})/(\partial_t\gamma) \rightarrow 0.
\end{equation*}

Moreover, the projection $\Phi_{\alpha}^{-1}(\gamma)\rightarrow [0,1)$ has compact fibers, see Lemma \ref{proper map}.
Using the compactness results of Lemma \ref{spherical bubble} and Lemma \ref{disc bubbles at t=1}, we deduce that the projection $\Phi_{\alpha}^{-1}(\gamma)\rightarrow [0,1)$ can be extended to a compact oriented $1$-dimensional manifold with boundary $\mathscr{M}_{\alpha}^{[0,1]}\rightarrow [0,1]$ such that 
\begin{gather*}
    \partial_0\mathscr{M}_{\alpha}^{[0,1]} = \mathscr{T}^{\chi}_{\mathbf{v}}(\alpha)\sqcup (\mathscr{M}_{\chi}\times\mathscr{M}^{s}_{\mathbf{z}_1}(Z,D,p,\alpha)),\quad \hbox{and}\\
    \partial_1\mathscr{M}_{\alpha}^{[0,1]} = \cup \bigg\{\bigtimes_{i=1}^N\bigtimes_{j=1}^{r_i} \mathscr{M}^{D_i}_{(\boldsymbol{\zeta}_i)_j}(L,\beta^i_j)\ |\   \alpha=\sum_{i,j}\beta_j^i\bigg\}.
\end{gather*}
Note that the (inward) boundary orientation on $\partial_0 \mathscr{M}_{\alpha}^{[0,1]}$ agrees with its Fredholm theoretic orientation described in \S 1.5.2 above, while the boundary orientation on $\partial_1\mathscr{M}_{\alpha}^{[0,1]}$ opposes its Fredholm theoretic orientation. This is because (i) we are orienting $\mathscr{M}_{\alpha}^{[0,1]}$ using $\partial_t$ in \eqref{orientation of homotopy moduli}, (ii) the trivialization of $\det(D_u)$ induced by the spin structure is compatible with gluing (see \cite[Lemma 8.3.5]{FO3-book2}), and (iii) $\partial_t$ points towards larger gluing parameters near $t=0$, but it points towards smaller gluing parameters near $t=1$. 

Next, we give a Gromov-Witten theoretic interpretation of the count of rational curves appearing in the moduli space $\partial_0 \mathscr{M}_{\alpha}^{[0,1]}$.

\begin{lemma}\label{gravitational-descendants}
The signed count of the elements of the $0$-dimensional moduli space  $\mathscr{M}^{s}_{\mathbf{z}_1}(Z,D,p,\alpha)$ (see (\ref{moduli of spherical bubbles})) is $ r_1!\dotsi r_N!\langle \psi_{r-2}\emph{\text{pt}} \rangle_{\alpha}$, where 
\begin{equation*}
    \langle \psi_{r-2}\emph{\text{pt}} \rangle_{\alpha} \defeq \int_{\overline{\mathscr{M}}_{0,1}(Z,\alpha)}c_1(\matheu{L})^{r-2}\wedge ev^*([p])
\end{equation*}
is the point Gromov-Witten descendant in the class $\alpha$. Recall that $\matheu{L}\rightarrow \overline{\mathscr{M}}_{0,1}(Z,\alpha)$ in this formula is the universal cotangent line bundle over the space of genus $0$ stable maps in the class $\alpha$ with $1$-marked point.
\end{lemma}

\begin{proof}
It is instructive to compare the moduli space $\mathscr{M}^{s}_{\mathbf{z}_1}(Z,D,p,\alpha)$ with
\begin{equation*}
\widehat{\mathscr{M}}^{s}_{\mathbf{z}_1}(Z,D,p,\alpha)= \{u:\mathbb{P}^1 \rightarrow Z|\ \overline{\partial}_J u = 0,\ u(\infty)=p,\ u(\zeta^{k})\in D,\ [u]=\alpha\}.    
\end{equation*}
These two moduli spaces describe essentially the same rational curves, except that the latter over-counts by a factor of $r!/(r_1!\dots r_N!)$.
The count $\# \widehat{\mathscr{M}}^{s}_{\mathbf{z}_1}(Z,D,p,\alpha)$ is independent of the choice of $(J,p,\mathbf{z}_1)$ as long as transversality is achieved. Consider the space 
\begin{equation*}
    \overline{\mathscr{M}}_{1+r}(Z,D,p,\alpha)=\{(u,w,(z_k)_{k=1}^r)\in \overline{\mathscr{M}}_{1+r}(Z,\alpha)\ |\ u(w)=p,\ u(z_k)\in D \},
\end{equation*}
where $\overline{\mathscr{M}}_{1+r}(Z,\alpha)$ is the space of genus $0$ stable maps with $1+r$ marked points in the class $\alpha$. Then,
\begin{equation}\label{formula 1}
    \#\widehat{\mathscr{M}}^{s}_{\mathbf{z}_1}(Z,D,p,\alpha) =\deg (\overline{\mathscr{M}}_{1+r}(Z,D,p,\alpha)\xrightarrow[]{\mathfrak{st}} \overline{\mathscr{M}}_{1+r}).
\end{equation}
The Deligne-Mumford space $\overline{\mathscr{M}}_{1+r}$ carries a universal line bundle $\matheu{L}_0$ which tracks the cotangent lines at the first marked point $w$. Recalling that $c_1(\matheu{L}_0)^{r-2}=\text{PD}(\text{pt})$, we deduce that
\begin{equation}\label{formula 2}
    \deg (\overline{\mathscr{M}}_{1+r}(Z,D,p,\alpha)\xrightarrow[]{\mathfrak{st}} \overline{\mathscr{M}}_{1+r}) = \int_{\overline{\mathscr{M}}_{1+r}(Z,D,p,\alpha)}c_1(\mathfrak{st}^*\matheu{L}_0)^{r-2}.
\end{equation}
At the same time, there is a forgetful map of degree $r!$,
\begin{equation*}
    \overline{\mathscr{M}}_{1+r}(Z,D,p,\alpha)\xrightarrow{\mathfrak{f}}\overline{\mathscr{M}}_{0,1}(Z,p,\alpha),
\end{equation*}
which forgets the extra marked points $z_1,\dots,z_r$. Observe that 
\begin{equation*}
\overline{\mathscr{M}}_{0,1}(Z,p,\alpha) = \mathscr{M}_{0,1}(Z,p,\alpha),
\end{equation*}
due to the constraint $r=c_1(\alpha)\leq j_Z$.
Moreover, in a nodal curve $C\in \overline{\mathscr{M}}_{1+r}(Z,D,p,\alpha)$, the component containing the marked point that is mapped to $p$ is never constant. It follows that
\begin{equation}\label{pullback is pullback}
    \mathfrak{f}^*\matheu{L} = \mathfrak{st}^*\matheu{L}_0.
\end{equation}
The lemma now follows by combining (\ref{formula 1}), (\ref{formula 2}), and (\ref{pullback is pullback}).
\end{proof}

\begin{remark}
A couple of remarks are in order:
\begin{itemize}
    \item[-] Lemma \ref{gravitational-descendants} probably holds as long as $\alpha$ is a spherical class such that $c_1(\alpha)=\alpha\cdot D$, without the assumption $c_1(\alpha)\leq j_Z$. The key identity (\ref{pullback is pullback}) in this case holds outside of a codimension $4$ subset, which should be enough.
    \item[-] In \cite{tonkonog-periods}, D.Tonkonog shows an interpretation of the gravitational descendant $\langle \psi_{r-2}\emph{\text{pt}}\rangle_{\alpha}$ using a certain moduli space of spheres which are tangent to a local hypersurface near \emph{pt}. That interpretation seems different from the one presented here.
\end{itemize}
\end{remark}
We have now collected all the ingredients for our main technical result.
\begin{theorem}\label{main theorem}
Let $L\subseteq (Z,I,\omega)$ be an admissible Lagrangian submanifold which is oriented and spin, such that $L\cap D=\emptyset$. Let $\alpha\in H_2(Z,L)$ be a relative class such that
\begin{equation*}
r\defeq\frac{1}{2}\mu_L(\alpha)=\alpha\cdot D\leq j_Z.    
\end{equation*}
Then,
\begin{equation}\label{main-formula}
    \frac{r!}{r_1!\dotsi r_N!}\tau^D_{\alpha}(L) + r!\langle \psi_{r-2}\emph{\text{pt}}\rangle_{\alpha} = W_L^r [\alpha].
\end{equation}
\end{theorem}

We now briefly discuss the case when $D$ is smooth and $c_1(Z)-D$ is nef, but not necessarily effective. In this setup, we also have an analogue of the counting formula \eqref{main-formula}, provided that $L\subseteq Z\backslash D$ satisfies an appropriate positivity assumption.

\begin{definition}\label{Maslov-positive-definition}
Let $D\subseteq (Z,I)$ be a smooth divisor. A Lagrangian embedding $L\subseteq Z\backslash D$ is said to be Maslov positive if
\begin{equation}\label{Maslov positive}
    \frac{1}{2}\mu_{L}(u) \geq \max\{u\cdot D, 1 \},
\end{equation}
for all $I$-holomorphic disc $u:\disk\rightarrow (Z,L)$.
\end{definition}

For example, if there is a smooth divisor $D_1$ such that $D\cup D_1$ is anti-canonical and $L\subseteq Z\backslash (D\cup D_1)$ is admissible, then $L\subseteq Z\backslash D$ is Maslov positive. More importantly, we have the following direct consequence of Gromov compactness.
\begin{lemma}
    If $L\subseteq Z\backslash D$ is Maslov positive and $J_z$ is a sufficiently small domain-dependent perturbation of $I$, then
    \begin{equation*}
    \frac{1}{2}\mu_{L}(u) \geq \max\{u\cdot D, 1 \}
    \end{equation*}
    for all $J_z$-pseudo-holomorphic discs $u:\disk\rightarrow (Z,L)$.
\end{lemma}
The methods leading up to Theorem \ref{main theorem} carry through, word-for-word, as long as one restricts to \emph{small} perturbations of $I$ when choosing a generic $J\in \mathscr{J}_N(Z,\omega)$ and a generic $(J_z)\in\mathscr{K}_J$. In a sense, this means that the enumerative invariants $m_{0,\beta}(L)$ and $\tau^D_{\alpha}(L)$ may depend on the complex structure $I$. 

The counting formula in this setup is
\begin{equation}\label{main-formula-smooth}
    \tau^D_{\alpha}(L) + r!\langle \psi_{r-2}\text{pt}\rangle_{\alpha} = (W^D_{L})^r [\alpha],
\end{equation}
where $\alpha\in H_2(Z,L)$ is a relative class and $W^D_L$ is the part of the super-potential which accounts for discs intersecting $D$,
\begin{equation*}
    W^D_L = \sum_{\mu_L(\beta)=2,\beta\cdot D=1} m_{0,\beta}(L) q^\beta.
\end{equation*}

Of course, in either counting formula (\ref{main-formula}) or (\ref{main-formula-smooth}), the quantity $r!\langle \psi_{r-2}\emph{\emph{\text{pt}}}\rangle_{\alpha}$ only contributes when $\alpha$ is a spherical class, i.e. $\partial \alpha = 0$ in $H_1(L)$. The sum over all such classes $\alpha$ is denoted by
\begin{equation*}
 r!\langle \psi_{r-2}\emph{\emph{\text{pt}}}\rangle_{r} \defeq \sum_{c_1(\alpha)=r, \alpha\in H_2(Z)} r!\langle \psi_{r-2}\emph{\emph{\text{pt}}}\rangle_{\alpha}.
\end{equation*}
It is the first potentially non-zero coefficient of the \emph{regularized quantum period} of $Z$, which we recall is given by
\begin{equation*}
    \widehat{G}_Z(t) = 1+\sum_{k\geq j_Z} k!\langle \psi_{k-2}\emph{\emph{\text{pt}}}\rangle_{k} t^k.
\end{equation*}
Quantum periods have been extensively studied in the literature due to their relevance in mirror symmetry. For instance, \cite{coates-et-al-quantum-periods} computes the regularized quantum periods of all Fano threefolds. Mirror symmetry for $Z$ predicts the existence of a Laurent polynomial $f$ in $n$-variables such that the constant term of the polynomial $f^k$ is the $k^{\text{th}}$-coefficient of $\widehat{G}_Z$. In \cite{tonkonog-periods}, D. Tonkonog shows that the super-potential $W_L$ of a \emph{monotone Lagrangian torus} $L$ satisfies this property. Mirror symmetry further predicts that $Z$ can be degenerated to a toric variety whose fan polytope is the Newton polytope of $f$. The pair $(X,f)$ is called a toric Landau-Ginzburg model for $Z$, see  \cite{katzarkov-przyj-old-and-new, coates-mirror-fano} for more literature on this mirror correspondence.  

\subsection{First applications} We now briefly mention a few (fairly direct) consequences of the formula (\ref{main-formula}). Recall that for any Lagrangian embedding $L\hookrightarrow Z$, the minimal Maslov number is defined as
\begin{equation*}
    N_L = \inf\{ \mu_L(\beta) \ | \beta\in H_2(Z,L) \}.
\end{equation*}
It was conjectured by Audin that monotone Lagrangian tori in $\mathbb{P}^n$ have minimal Maslov number $2$. This conjecture was proved by Cieleback-Mohnke in \cite{CM-neck-stretching} using an SFT neck-stretching arguments. The quantum period theorem of Tonkonog (see \cite{tonkonog-periods}) extends this result to monotone Lagrangian tori in other Fano varieties.

\begin{corollary}\label{Audin's conjecture}
Let $Z$ be a Fano variety such that $\langle \psi_{r-2}\emph{\text{pt}}\rangle_{r}\neq 0$, where $r$ is the pseudo-index of $Z$. Let $D\subseteq Z$ be an anti-canonical divisor which is simple normal crossings. If $L\subseteq Z\backslash D$ is an admissible Lagrangian of non-positive sectional curvature, then $N_L=2$. 
\end{corollary}

We note that the constraint $\langle \psi_{r-2}\text{pt}\rangle_{r}\neq 0$ automatically excludes Fano varieties which contain lines (i.e. those with $r=j_Z=1$). However, for such Fano manifolds, the conclusion of the corollary tautologically holds for any Lagrangian embedding $L\hookrightarrow Z$. Among Fano varieties with $j_Z\geq 2$, the constraint $\langle \psi_{r-2}\text{pt}\rangle_{r}\neq 0$ does not always hold, but at the same time it doesn't seem very restrictive. For instance, Corollary \ref{Audin's conjecture} applies to all Fano surfaces and threefolds.
\vspace{5pt}\\
\textit{Proof sketch.}
Requiring that $L$ admits a metric of non-positive curvature ensures that $\tau_{\alpha}(L)=0$ whenever $\partial \alpha = 0$ in $\pi_1(L)$ and $c_1(\alpha)=r$. Indeed, if $\tau_{\alpha}(L)\neq 0$, then the tangency moduli space can't be empty for any choice of almost complex structure $J\in \mathscr{J}_N(Z,\omega)$. In particular, one gets a disc $u_k$ that is fully tangent to $D$ (i.e. $j^D_{0,r-1}(u_k) = 0$) for each of the neck-stretched almost complex structures $J_k$ constructed in \cite{CM-neck-stretching}. The SFT limit of the sequence $(u_k)$ is a holomorphic building as described in \cite[corollary 2.9]{CM-neck-stretching}. Examining the top part of the building, $Z\backslash L$, only the component tangent to $D$ can be non-constant. Therefore, the bottom part of the building is a half-cylinder in $T^*L$ with boundary on $L$, which is asymptotic to a Reeb orbit. But Reeb orbits in $T^*L$ are lifts of closed geodesics from $L$, none of which is null-homotopic.

Once we known that $\tau_{\alpha}(L)=0$, and $\langle \psi_{r-2}\text{pt}\rangle_{r}\neq 0$, the claim follows immediately from the counting formula (\ref{main-formula}).
\qed
\vspace{5pt}

Our next application is towards Lagrangian topology, along the same lines of Fukaya's work in \cite{Fukaya-icm}. Recall that the index of a Fano manifold is the largest integer $i_Z$ such that
\begin{equation*}
    c_1(Z) = i_Z H
\end{equation*}
for some primitive class $H\in H^2(Z,\mathbb{Z})$. For example, $i_{\mathbb{P}^n}=n+1$.

\begin{corollary}
Suppose that $Z$ is a Fano variety of index $r\geq 2$ such that $\langle \psi_{r-2}\emph{\text{pt}}\rangle_{r}\neq 0$. Let $H_1,\dots,H_r$ be a collection of homologous divisors in general position whose union $D=\cup_{i=1}^r H_i$ is anti-canonical. If $L\subseteq Z\backslash D$ is an admissible Lagrangian of non-positive sectional curvature, then $L$ is finitely covered by a product $(S^1)^{r-1}\times K$. 
\end{corollary}

\vspace{5pt}
\textit{Proof sketch.}
Let $\alpha \in H_2(Z)$ be a curve class which has a non-zero contribution to the Gromov-Witten descendant $\langle \psi_{r-2}\text{pt}\rangle_{r}$.  As in the proof of corollary \ref{Audin's conjecture}, we first use the non-positive curvature assumption to ensure that $\tau_{\alpha}(L)= 0$. Using the counting formula (\ref{main-formula}), the class $\alpha$ decomposes into disc classes $\beta_1,\dots,\beta_r\in H_2(Z,L)$ such that $\mu_L(\beta_i)=2$, $\beta_i\cdot H_j = \delta_{ij}$, and $m_{0,\beta_i}(L)\neq 0$. We now study properties of the loops $\gamma_i\defeq \partial \beta_i\in \pi_1(L)$.

\vspace{4pt}
\textit{Claim:} In $H_1(L)$, the loops $\gamma_i$ generate a free Abelian group of rank $r-1$.

\textit{Proof of Claim.} Indeed, if we have a relation $a_1\gamma_1+\dotsi+a_r\gamma_r = 0$, then there is a curve class $C\in H_2(Z)$ such that
\begin{equation*}
    C = a_1\beta_1+\dotsi+a_r\beta_r \quad\text{in}\ H_2(Z,L).
\end{equation*}
It follows that $C\cdot H_j = a_j$, for all $j=1,\dots,r$. But recall that all the divisors $H_j$ are homologous to one another. Therefore, the classes $(\gamma_i)$ only satisfy one non-trivial relation in $H_1(L)$, which is $\gamma_1+\dotsi+\gamma_r=0$.
\vspace{4pt}

Going back to (the proof of) Lemma \ref{count agrees with classical count}, the integer $m_{0,\beta_i}$ is the degree of an evaluation map
\begin{equation*}
    ev: \mathscr{M}_{H_i}(L,\beta_i)\rightarrow L,
\end{equation*}
which factors through the component $\mathcal{L}_{\gamma_i}(L)$ of the free loop space of $L$. In particular $H_n(\mathcal{L}_{\gamma_i}(L))\neq 0$. Since $L$ has non-positive sectional curvature, this implies that the centralizers
\begin{equation*}
    C_{\gamma_i} = \{g\in\pi_1(L) \ |\ g^{-1}\gamma_i g = \gamma_i \}
\end{equation*}
have finite index in $\pi_1(L)$, see \cite[Lemma 2.4]{Fukaya-icm}. Therefore, their intersection $C=\cap_{i=1}^r C_{\gamma_i}$ is also a finite index subgroup of $\pi_1(L)$. Let $d$ be its index, and let $G$ be the subgroup of $\pi_1(L)$ that is generated by the collection $\gamma_i^d$. Then $G$ is a free Abelian subgroup of rank $r-1$ (see Claim above) and its centralizer $C_G\subseteq \pi_1(L)$ has finite index (note that $C_{\gamma_i}\subseteq C_{\gamma_i^d}$). Let $L^+$ be the finite covering of $L$ whose fundamental group is $C_G$. Then $L^+$ in turn supports a metric of non-positive sectional curvature. Moreover, $G$ is a free Abelian subgroup of the center $\mathcal{Z}(\pi_1(L^+))$. Using the center theorem of Lawson-Yau (see \cite{Lawson-Yau}), we deduce that $L^+$ (and hence $L$) is covered by a product $(S^1)^{r-1}\times K$ for some smooth compact manifold $K$.
\qed
\vspace{5pt}

\section{Lagrangian tori in cyclic covers}
\subsection{Preliminaries} Let $(Y,I_Y,\omega_Y)$ be a smooth Fano variety of dimension $n\geq 2$ which is equipped with K\"ahler form $\omega_Y$.
\begin{definition}
A cyclic covering datum is a triple $(r,\mathscr{L},\sigma)$, where $r\geq 2$ is an integer, $\mathscr{L}\rightarrow Y$ is an ample line bundle, and $\sigma:\mathscr{O}_Y\rightarrow \mathscr{L}^r$ is a holomorphic section whose zero locus $D_Y=\sigma^{-1}(0)\subseteq Y$ is transversely cut.
\end{definition}
As the name suggests, a cyclic covering datum $(r,\mathscr{L},\sigma)$ determines a projective variety $X$ and a cyclic covering map $\phi:X\rightarrow Y$, where
\begin{equation*}
   X = \left\{ p\in\mathscr{L} \ | \ p^{\otimes r} = \sigma\right\}\subseteq \text{Tot}(\mathscr{L}).
\end{equation*}
More precisely,
\begin{equation}\label{the fano cover}
    X = \spec_{\mathscr{O}_Y}\left(\mathscr{O}_Y\oplus \mathscr{L}^{-1}\oplus \dotsi\oplus (\mathscr{L}^{-1})^{\otimes (r-1)}\right).
\end{equation}
In a local affine chart $\spec(A)\subseteq Y$, the section $\sigma$ may be regarded as a local function $f\in A$. The covering map is then modeled by
\begin{equation}\label{local-model}
    \spec\left(A[t]/(t^r-f)\right) \rightarrow \spec(A).
\end{equation}
In particular, $X$ is a smooth projective variety. The ramification divisor of $\phi$ satisfies the linear equivalence (see (\ref{local-model}))
\begin{equation}\label{ramification locus}
    rR = (r-1)\phi^{-1}(D_Y),
\end{equation}
where $\phi^{-1}(D_Y)$ is the (non-reduced) pre-image of the branch locus. We denote by $D_X$ the reduced form of the ramification locus. It corresponds to $(t=0)$ in the local model (\ref{local-model}). By the Riemann-Hurwitz formula, we have that
\begin{align}\label{closed riemann hurwitz}
    c_1(X) &= \phi^*c_1(Y) - R \\
           &= \phi^*\left(c_1(Y)-(r-1)\mathscr{L} \right).\notag
\end{align}
Note in particular that if $D_Y$ is anti-canonical, then so is $D_X$.
\begin{corollary}
If $c_1(Y) - D_Y$ is nef, then $X$ is Fano.
\end{corollary}
\begin{proof}
Using the Riemann-Hurwitz formula as in (\ref{closed riemann hurwitz}),
\begin{equation*}
    rK_{X}^{-1} = \phi^*\left(K_Y^{-1}+(r-1)(K_Y^{-1}-D_Y) \right).
\end{equation*}
The divisor $K_Y^{-1}+(r-1)(K_Y^{-1}-D_Y)$ is ample by the Nakai-Moishezon criterion. Its pullback by the finite map $\phi$ is therefore ample.
\end{proof}

Let $L_Y\subseteq Y\backslash D_Y$ be a Lagrangian torus. Denote by $L_X = \phi^{-1}(L_Y)\subseteq X$ its pre-image.
\begin{proposition}\label{Kahler form}
For each neighborhood $U\subseteq X\backslash L_X$ of $D_X$, there is a function $\rho:X\rightarrow \mathbb{R}$ with compact support in $U$ such that the $2$-form
\begin{equation}
    \omega_X = \phi^*\omega_Y + dd^c\rho
\end{equation}
is K\"ahler. In particular, $L_X\subseteq (X,\omega_X)$ is Lagrangian.
\end{proposition}
\begin{proof}
If $[\omega_Y]\in H^2(Y,\mathbb{R})$ is rational, i.e. the curvature of some \emph{ample} line bundle on $Y$ (up to a factor of $i/2\pi$), then this is just Lemma 2.15 in \cite{mypaper1}: one exploits the fact that a finite pull-back of an ample line bundle is ample.

More generally, since $Y$ is Fano, $H^2(Y,\mathscr{O}_Y)=0$ so that any K\"ahler form is a convex combination of rational K\"ahler forms.
\end{proof}

The Lagrangian $L_X$ is a covering of the torus $L_Y$, so it is a disjoint union of tori. In the examples we consider, $L_X$ will always be connected. This is due to the following result.
\begin{proposition}\label{connected pre-image}
Suppose that $L_Y$ bounds a (topological) disc $v$ such that $v\cdot D_Y = 1$. Then $L_X\subseteq X$ is connected.
\end{proposition}
\begin{proof}
Consider the case when $n>2$, and let $\overset{\circ}{N}_{\epsilon} = B_{\epsilon}(D_Y)\backslash D_Y$ be a punctured neighborhood of $D_Y$. Since $Y$ is Fano, it must be simply connected. Therefore, the map
\begin{equation*}
    \pi_1(\overset{\circ}{N}_{\epsilon})\rightarrow \pi_1(Y\backslash D_Y)
\end{equation*}
is surjective. By the Lefschetz hyperplane theorem, $D_Y$ is simple connected. Therefore, there is surjective map
\begin{equation*}
    \pi_1(S^1)\rightarrow \pi_1(\overset{\circ}{N}_{\epsilon})
\end{equation*}
coming from the homotopy long exact sequence for the bundle $\overset{\circ}{N}_{\epsilon}\rightarrow D_Y$. It follows that $\pi_1(Y\backslash D_Y)$ is Abelian.

If $n=2$,  $\pi_1(Y\backslash D_Y)$ is still Abelian (due to work of Zariski \cite{zariski-branched-covers}) though it is far less trivial. See \cite[Theorem II]{nori} for an explicit statement.

Since it is Abelian, the fundamental group $\pi_1(Y\backslash D_Y)$ can be computed using Poincar\'e duality:
\begin{align*}
    \pi_1(Y\backslash D_Y) &= H_1(Y\backslash D_Y,\mathbb{Z}) = H^{2n-1}(Y,D_Y,\mathbb{Z})\\
    &= \text{coker}(H^{2n-2}(Y,\mathbb{Z})\rightarrow H^{2n-2}(D_Y,\mathbb{Z})).
\end{align*}
Let $q=\inf\{ \abs{\beta\cdot D_Y} \ |\ \beta\in H_2(Y,\mathbb{Z})\}$. We therefore get an isomorphism
\begin{equation*}
    \text{lk}: \pi_1(Y\backslash D_Y) \rightarrow \mathbb{Z}_q,\quad \gamma \mapsto v_\gamma\cdot D_Y,
\end{equation*}
where $v_\gamma$ is any disc map whose boundary is $\gamma$. Since $L_Y$ bounds a disc $v$ such that $\text{lk}(v)=1$, the group homomorphism
\begin{equation*}
    \pi_1(L_Y) \rightarrow \pi_1(Y\backslash D_Y)
\end{equation*}
is surjective. It follows that $L_Y$ has a connected pre-image in any \emph{unbranched} covering of $Y\backslash D_Y$.
\end{proof}
Our goal is to relate counts of Maslov index $2$ discs associated with the pairs $(X,L_X)$ and $(Y,L_Y)$. Recall the following open analogue of the Riemann-Hurwitz formula.

\begin{lemma}\label{Riemann Hurwitz}
Let $u:\disk\rightarrow (X,L_X)$ be a disc map and $v\defeq \phi\circ u :\disk\rightarrow (Y,L_Y)$ its pushforward. Then,
\begin{equation*}
    \frac{1}{2}\mu_{L_X}(u) = \frac{1}{2}\mu_{L_Y}(v) - \frac{r-1}{r} v\cdot D_Y.
\end{equation*}
\end{lemma}

\begin{proof}
Let $s$ be a generic section of the line bundle $\wedge_{\mathbb{C}}^n u^*TX$ whose restriction $s_{|\partial\mathbb{D}}$ to the boundary is a nowhere vanishing section of the subbundle $\wedge_{\mathbb{R}}^n u_{|\partial\mathbb{D}}^*TL_X$.
Then $1/2\mu_{L_X}(u)$ is the signed count of the zeroes of $s$. Similarly, $1/2\mu_{L_X}(v)$ is the signed count of zeroes of the section $\phi_*(s)\defeq u^*(\wedge^n d\phi)\circ s$ of the bundle pair $(\wedge_{\mathbb{C}}^n v^*TY,\wedge_{\mathbb{R}}^n v^*TL_Y)$. By definition, $\wedge^n d\phi$ is a section of the line bundle $(\wedge^n TX)^{-1} \otimes \wedge^n \phi^*TY$ whose zero set (with multiplicity) is the ramification locus $R$. It follows that
\begin{equation*}
    \frac{1}{2}\mu_{L_X}(u) = \frac{1}{2}\mu_{L_Y}(v) - u\cdot R.
\end{equation*}
The Lemma now follows from the computation of $R$ in (\ref{ramification locus}).
\end{proof}
Note the following special case of the previous construction.
\begin{corollary}\label{monotone case}
Suppose that $D_Y\subseteq Y$ is anti-canonical, and that the Maslov class of $L_Y\subseteq Y\backslash D_Y$ is trivial. Then $D_X\subseteq X$ is anti-canonical, and the Maslov class of $L_X\subseteq X\backslash D_X$ is trivial as well. If $L_Y\subseteq (Y,\omega_Y)$ is also monotone, then so is $L_X\subseteq (X,\omega_X)$.
\end{corollary}
\begin{proof}
Recall (see \eqref{Maslov identity}) that $L_Y\subseteq Y\backslash D_Y$ has a trivial Maslov class if and only if 
\begin{equation*}
    \frac{1}{2}\mu_{L_Y}(v) = v\cdot D_Y, \quad\hbox{for all}\ v:\disk\rightarrow(Y,L_Y).
\end{equation*}
Using Lemma \ref{Riemann Hurwitz}, the identity in (\ref{ramification locus}), and the fact that $R=(r-1)D_X$, we deduce that
\begin{equation*}
\frac{1}{2}\mu_{L_X}(u) = \frac{1}{r}v\cdot D_Y  =  u\cdot D_X.
\end{equation*}
It follows that $L_X\subseteq X\backslash D_X$ has a trivial Maslov class. 

If $L_Y\subseteq (Y,D_Y)$ is monotone, then there is a constant $\lambda$ such that
\begin{equation*}
    \frac{\lambda}{2}\mu_{L_Y}(v) = \int_{\mathbb{D}} v^*\omega_Y \quad\hbox{for all}\ v:\disk\rightarrow(Y,L_Y).
\end{equation*}
It follows (using Proposition \ref{Kahler form}) that for any disc map $u:\disk\rightarrow (X,L_X)$,
\begin{equation*}
      \int_{\mathbb{D}} u^*\omega_X =\frac{\lambda r}{2}\mu_{L_X}(u).
\end{equation*}
\end{proof}
\subsection{Super-potentials} 
From now on, we assume that the Lagrangian $L_Y\subseteq Y\backslash D_Y$ is Maslov positive. Recall this positivity means that
\begin{equation}\label{Maslov positive-2}
    \frac{1}{2}\mu_L(v)\geq \max\{1,v\cdot D_Y\}
\end{equation}
for all $I_Y$-holomorphic discs $v:\disk\rightarrow (Y,L_Y)$. It ensures that there are no non-constant $J$-holomorphic discs of Maslov index $0$ for sufficiently small perturbations $J\in\mathscr{J}_N(Y,\omega_Y)$ of $I_Y$. Therefore, one can define the numerical invariants
\begin{equation}\label{maslov index 2 counts}
    m_{0,\beta}(L_Y) = \deg\left(\mathscr{M}^J_{0,1}(L_Y,\beta)\xrightarrow[]{ev} L_Y \right)
\end{equation}
as before, where $\mathscr{M}^{J}_{0,1}(L_Y,\beta)$ is the moduli-space of (unparametrized) $J$-holomorphic discs in the class $\beta\in H_2(Y,L_Y)$ with $1$ boundary marked point. See the proof of Lemma \ref{count agrees with classical count} for a discussion about regularity and compactness of this moduli space.

\begin{remark}
Note that, \`a-priori, the integers $m_{0,\beta}(L_Y)$ depend on $I_Y$: if a perturbation $(J_t)_{0\leq t\leq 1}$ is large, then $J_1$ may allow for pseudo-holomorphic discs with Maslov index $0$. This issue does not arise in the examples we study: our choices of $L_Y$  will be \emph{admissible} with respect to an anti-canonical divisor of the form $D_Y\cup \widehat{D}$.
\end{remark}

Recall that the super-potential associated with $L_Y$ is \begin{equation*}
    W_{L_Y} = \sum_{\mu_{L_Y}(\beta)=2} m_{0,\beta}(L_Y)q^{\partial \beta},
\end{equation*}
which we view as a function $W_{L_Y}:\spec(\mathbb{C}[H_1(L_Y,\mathbb{Z})])\rightarrow \mathbb{C}$. When $L_Y$ is a torus, $W_{L_Y}$ is a Laurent polynomial. Our goal is to relate the super-potentials associated with $L_Y$ and its pre-image $L_X$ under the cyclic covering map $\phi:X\rightarrow Y$.

Maslov positivity is inherited through cyclic covers. Indeed, the Riemann-Hurwitz formula from Lemma \ref{Riemann Hurwitz} can be rewritten as
\begin{equation}\label{Riemann Hurwitz 2}
    \frac{1}{2}\mu_{L_X}(u)-u\cdot D_X = \frac{1}{2}\mu_{L_Y}(v)-v\cdot D_Y.
\end{equation}
The following is therefore an immediate consequence.
\begin{corollary}\label{maslov positivity preserved under cyclic coverings}
$L_Y\subseteq Y\backslash D_Y$ is Maslov-positive if and only if $L_X\subseteq X\backslash D_X$ is Maslov-positive.
\end{corollary}
\begin{proof}
Suppose $L_X\subseteq X\backslash D_X$ is Maslov-positive (the other direction is straightforward). Let $v$ be an $I_Y$-holomorphic disc. Then, the map
\begin{equation*}
    \disk\rightarrow (Y,L_Y):\quad z\mapsto v(z^r)
\end{equation*}
has an $I_X$-holomorphic lift $u:\disk\rightarrow (X,L_X)$. Using (\ref{Riemann Hurwitz 2}),
\begin{equation*}
   r\left( \frac{1}{2}\mu_{L_Y}(v)-v\cdot D_Y\right)= \frac{1}{2}\mu_{L_X}(u)-u\cdot D_X \geq 0.
\end{equation*}
If $v\cdot D_Y=0$, then $v$ itself has a lift $u$ and so $\mu_{L_Y}(v)=\mu_{L_X}(u)\geq 2$.
\end{proof}

Let $\beta\in H_2(X,L_X)$ be a class of Maslov index $2$ such that the moduli space $\mathscr{M}^{I_X}_{0,1}(L_X,\beta)$ is non-empty. Set
\begin{equation*}
\alpha = \phi_*(\beta)\in H_2(Y,L_Y,\mathbb{Z}).
\end{equation*}
By the inequality (\ref{Maslov positive-2}), $\beta\cdot D_X \in \{0,1\}$.   

\begin{lemma}\label{count up is count down 1}
Suppose that $\beta\cdot D_X = 0$. Then,
\begin{equation*}
 m_{0,\beta}(L_X)=m_{0,\alpha} (L_Y)   .
\end{equation*}
\end{lemma}
\begin{proof}
Since $\mu_{L_Y}(\alpha)=2$, there is an almost complex structure $J_Y$ on $Y$ such that $\mathscr{M}^{J_Y}_{0,1}(L_Y,\alpha)$ is Fredholm regular and
\begin{equation}\label{J trivial near divisor}
J_Y=I_Y \quad \hbox{near } D_Y.
\end{equation}
The condition (\ref{J trivial near divisor}) ensures that $J_X=\phi^{-1}J_Y$ extends through the ramification locus. Since $\beta\cdot D_X = 0$, the map
\begin{equation}\label{pushforward without hamiltonian perturbations}
    \mathscr{M}^{J_X}_{0,1}(L_X,\beta)\rightarrow \mathscr{M}^{J_Y}_{0,1}(L_Y,\alpha):\quad u\mapsto \phi\circ u
\end{equation}
is an orientation preserving proper local diffeomorphism. Hence, it is a covering map of degree $r$. The lemma now follows from the commutative diagram
\begin{center}
\begin{tikzcd}
\mathscr{M}^{J_X}_{0,1}(L_X,\beta)\arrow[d] \arrow[r, "ev"] & L_X \arrow[d,"\phi"]\\
\mathscr{M}^{J_Y}_{0,1}(L_Y,\alpha) \arrow[r,"ev"] & L_Y.
\end{tikzcd}
\end{center}
\end{proof}
If $\beta\cdot D_X = 1$, its pushforward satisfies the equations
\begin{equation*}
    \mu_{L_Y}(\alpha)=2r \quad \hbox{and}\quad \alpha\cdot D_Y = r.
\end{equation*}
The image of the map (\ref{pushforward without hamiltonian perturbations}) lands in the subspace of $\mathscr{M}^{J_Y}_{0,1}(L_Y,\alpha)$ consisting of discs that are tangent to the divisor $D_Y$ to order $r$.

\begin{lemma}\label{count up is count down 2}
Suppose that $\beta\cdot D_X = 1$. Then, 
\begin{equation*}
m_{0,\beta}(L_X)=\tau_{\alpha}(L_Y).    
\end{equation*}
\end{lemma}
\begin{proof} 
Let $(J_Y,J_{Y,z})$ be a small perturbation of $I_Y$ such that the moduli space $\mathscr{T}_r(Y,L_Y,J_Y,J_{Y,z},\alpha)$ (see Remark \ref{remark about notation} for notation) is Fredholm regular and let $(J_X,J_{X,z}) = (\phi^{-1}J_Y,\phi^{-1}J_{Y,z})$.  Then, the moduli space $\mathscr{T}_0(X,L_X,J_X,J_{X,z},\beta)$ is Fredholm regular (\cite[Proposition 3.14]{mypaper1}) and the map
\begin{equation*}
   \mathscr{T}_1(X,L_X,J_Y,J_{Y,z},\beta)\rightarrow  \mathscr{T}_r(Y,L_Y,J_X,J_{X,z},\alpha):\quad u\mapsto \phi\circ u
\end{equation*}
is an $r$-fold covering map (\cite[Lemma 3.11]{mypaper1}). It follows that
\begin{equation}\label{won't use it again}
    \tau_{\alpha}(L_Y) = \deg(\mathscr{T}_1(X,L_X,J_X,J_{X,z},\beta)\xrightarrow[]{ev_\chi} L_X) .
\end{equation}
Let $J$ be a small perturbation of $I_X$ such that discs of Maslov index $2$ are Fredholm regular. A generic path from $J_{X,z}$ to $J$ provides an equivalence between the pseudo-cycle in (\ref{won't use it again}) and the pseudo-cycle
\begin{equation*}
    \mathscr{T}_1(X,L_X,J,J,\beta)\xrightarrow[]{ev_\chi} L_X.
\end{equation*}
Since $\beta\cdot D_X=1$, the tautological map
\begin{equation*}
  \iota_{\chi}:  \mathscr{T}_1(X,L_X,J,J,\beta)\rightarrow \mathscr{M}^{J}_{0,1}(L,\beta)
\end{equation*}
is a diffeomorphism. The lemma hence follows because $ev_{\chi}=ev\circ\iota_{\chi}$.
\end{proof}
Motivated by the dichotomy of Lemmas \ref{count up is count down 1} and \ref{count up is count down 2}, we define
\begin{equation*}
    W^{Z\backslash D}_L = \sum_{\mu(\beta)=2,\beta\cdot D=0}m_{0,\beta}(L)q^{\partial \beta}\quad \text{and}\quad W^D_L = W_L- W^{Z\backslash D}_L.
\end{equation*}
Note that the map $\beta\mapsto\alpha$ extends to a homomorphism of algebras
\begin{equation*}
    \phi_* :\mathbb{C}[H_1(L_X,\mathbb{Z})]\rightarrow \mathbb{C}[H_1(L_Y,\mathbb{Z})].
\end{equation*}
This homomorphism describes an \emph{unbranched} cyclic covering of degree $r$,
\begin{equation*}
    \phi^{\vee}: \spec(\mathbb{C}[H_1(L_Y,\mathbb{Z})])\rightarrow \spec(\mathbb{C}[H_1(L_X,\mathbb{Z})]).
\end{equation*}
\begin{theorem}\label{super-potential-theorem}
Let $(Y,I_Y,\omega_Y)$ be a Fano K\"ahler triple, $\phi:X\rightarrow Y$ a cyclic covering map of degree $r$ branched along a smooth ample divisor $D_Y\subseteq Y$. Let $L_Y\subseteq Y$ be a Maslov-positive Lagrangian torus in the complement of $D_Y$ such that $W_{L_Y}^{D_Y}\neq 0$. Then, the pre-image $L_X=\phi^{-1}(L_Y)$ is a connected torus. Moreover, the cohomology class $[\phi^*\omega_Y]\in H^2(X,L_X)$ supports a K\"ahler form $\omega_X$ on $X$ such that $L_X\subseteq (X,\omega_X)$ is Lagrangian. Its associated super-potential is given by the formula
\begin{equation}\label{super-potential-formula}
    (\phi^{\vee})^*W_{L_X} = W^{Y\backslash D_Y}_{L_Y} + (W_{L_Y}^{D_Y})^r - r!\langle\psi_{r-2}\emph{\text{pt}} \rangle_r.
\end{equation}
\end{theorem}
\begin{proof}
The pre-image $L_X$ is connected because of Proposition \ref{connected pre-image} and the assumption that $W_{L_Y}^{D_Y}\neq 0$. The K\"ahler form $\omega_X$ is described in Proposition \ref{Kahler form}. Finally, the super-potential formula follows from Lemmas \ref{count up is count down 1}-\ref{count up is count down 2} and Theorem \ref{main theorem}. Note that the requirement $r\leq j_Y$ is automatic. Indeed,  if $C$ is a curve with Chern number $j_Y$, then
\begin{equation*}
    j_Y = K_Y^{-1}\cdot C \geq D_Y\cdot C\geq r, 
\end{equation*}
because $D_Y$ is ample and divisible by $r$.
\end{proof}
\begin{remark}
It is useful to remember the interpretation of $r!\langle\psi_{r-2}\emph{\text{pt}}\rangle_r$ as the count of rational curves $u:\mathbb{P}^1\rightarrow Y$ of degree $r\leq j_Y$ such that $u(\infty)=p$ and $u(\zeta^i)\in D_Y$. In particular, this term does not contribute if $r<j_Y$.
\end{remark}

\begin{remark}\label{remark about abelian covers}
Analogues of Theorem \ref{super-potential-theorem} hold for more complicated Abelian covers which are branched along simple normal crossings divisors in $Y$, such as the $\mathbf{a}$-branched covers studied by N. Sheridan, see Definition 3.5.16 in \cite{sheridan-cy}. Such covers can be constructed by taking iterated cyclic coverings branched over components of an anti-canonical simple normal crossings divisor. See the example of the Del Pezzo surface $\text{Bl}_5(\mathbb{P}^2)$ below.
\end{remark}
\subsection{Weak LG-models}
In \cite{Fano-search}, the authors outline a program to search for new Fano 4-folds inspired by closed string mirror symmetry. The idea is to establish a database for known Fano $4$-folds, their quantum periods (this is done in \cite{database-for-fourfolds}), and their associated \emph{weak LG-models}; see \cite{database-for-threefolds} for what the end product looks like for Fano $3$-folds.

\begin{definition}
Let $X$ be a Fano manifold of dimension $n$ and let $f$ be a Laurent polynomial in $n$ variables. We say that $f$ is a weak LG-model for $X$ if 
\begin{equation*}
    c_0(f^k) = k!\langle \psi_{k-2}\emph{\text{pt}}\rangle_{k}\quad\text{for all}\ k\geq 1,
\end{equation*}
where $c_0(f^k)$ is the constant term of the Laurent polynomial $f^k$.
\end{definition}

Given such a database, one can compare any newly constructed Fano manifold to this database and decide if it is genuinely new. The idea being that if $X_1$ and $X_2$ are deformation equivalent Fano manifolds, then their quantum periods agree $\widehat{G}_{X_1}(t)=\widehat{G}_{X_2}(t)$. This principle was applied successfully in \cite{New-Fanos} to produce new Fano $4$-folds. It is worth noting that $\widehat{G}_X$ appears to be a strong invariant of $X$: in dimensions $n\leq 3$ where we already have a classification, it is a complete invariant of deformation classes of Fano manifolds.

The bridge between this program and our work is the quantum period theorem of Tonkonog.
\begin{theorem}\label{quantum period theorem}(see \cite{tonkonog-periods})
Let $Y$ be a Fano manifold and let $L\subseteq Y$ be a monotone Lagrangian torus. Then, the super-potential $W_L$ is a weak LG-model for $Y$.
\end{theorem}
The Fano manifold $Y$ above is equipped with an anti-canonical K\"ahler form $\omega_Y$. Recall that any such K\"ahler form arises as the curvature of some (positive) metric $h$ on the line bundle $K_Y^{-1}$,
\begin{equation*}
    \omega_Y = -\frac{1}{4\pi}dd^c\log\abs{s}_h,
\end{equation*}
where $s$ is a holomorphic section of $K_Y^{-1}$. In particular, in the complement of the anti-canonical divisor $D=s^{-1}(0)$, the K\"ahler form $\omega_Y$ has a \emph{preferred} primitive
\begin{equation*}
\theta = -\frac{1}{4\pi}d^c\log\abs{s}_h.  
\end{equation*}
Unless otherwise specified, this is always the exact structure given to $Y\backslash D$.
\begin{corollary}\label{LG-model for cyclic covers}
Let $Y$ be a Fano manifold of index $r$, $D\subseteq Y$ a smooth anti-canonical divisor such that $Y\backslash D$ contains a graded exact Lagrangian torus. Let $X$ be the $r$-fold cyclic covering of $Y$ which is branched along $D$. If $f_Y$ is a weak LG-model for $Y$, then $f_Y^r-c_0(f^r_Y)$ is a weak LG-model for $X$.
\end{corollary}

\begin{proof}
Just observe that the exact graded Lagrangian torus $L\subseteq Y\backslash D$ is both admissible so that Theorem \ref{super-potential-theorem} applies and monotone so that Theorem \ref{quantum period theorem} applies. To see that $L$ is monotone, one needs the integral formula
\begin{equation*}
   \int_{\mathbb{D}}u^*\omega - \int_{\partial \mathbb{D}} u^*\theta = u\cdot D,
\end{equation*}
which holds for any disc map $u:\mathbb{D}\rightarrow Y$ such that $u(\partial \mathbb{D})\subseteq Y\backslash D$.
\end{proof}

We note that $f_X^{\text{nv}}\defeq f_Y^r-c_0(f^r_Y)$ is really not the correct mirror to $X$. Suppose that $f_Y:(\mathbb{C^*})_Y^n\rightarrow \mathbb{C}$ is a genuine super-potential (i.e. associated with a graded exact Lagrangian torus $L_Y$), where $(\mathbb{C^*})_Y^n = \spec(\mathbb{C}[H_1(L_Y,\mathbb{Z})])$ is just a copy of $(\mathbb{C^*})^n$. Then, Theorem \ref{super-potential-theorem} suggests that the super-potential $f_X$ associated with $X$ acts on a \emph{quotient} of $(\mathbb{C^*})_Y^n$. This means that there is a free action of $\mathbb{Z}_r$ on $(\mathbb{C}^*)_Y^n$ which fixes $f_X^{\text{nv}}$. The correct mirror is the quotient 
\begin{equation}\label{corrected mirror}
 f_X :(\mathbb{C^*})_Y^n/\mathbb{Z}_r\xrightarrow[]{[f_X^{\text{nv}}]} \mathbb{C}.   
\end{equation}

\begin{remark} \label{correct mirror is quotient}
Corollary \ref{LG-model for cyclic covers} is purely an algebro-geometric result, but our methods require the existence of a graded exact Lagrangian torus. Note further that it applies the same to $m$-fold cyclic coverings of $Y$ branched along $D$, as long as $m$ divides the index $r$.
\end{remark}
\begin{conjecture}
Corollary \ref{LG-model for cyclic covers} holds without the assumption that $Y\backslash D$ contains a graded exact Lagrangian torus.
\end{conjecture}

\subsection{Fano Hypersurfaces in $\mathbb{P}^{n+1}$}
We now present an application of the super-potential formula (\ref{super-potential-formula}) to the Fukaya category of a degree $d< n+1$ Fano hypersurface $X_{d}\subseteq \mathbb{P}^{n+1}$. The index $1$ Fano case, which corresponds to $d=n+1$, was the subject of \cite{mypaper1}. The methods we use here apply the same way to the index $1$ Fano case as well, but we exclude it for clarity of exposition. We use the notation $[t:x_0:\dotsi:x_n]$ for homogeneous coordinates in the projective space $\mathbb{P}^{n+1}$. The Fano hypersurface $X_d$ can be cut-out by an equation of the type
\begin{equation*}
    X_d = V(t^d-x_0\dots x_{d-1} +\epsilon f(x_0,\dots,x_n)),\quad \abs{\epsilon}\hbox{ small, } f \hbox{ generic.}
\end{equation*}
It comes with a $d$-to-$1$ cyclic covering map
\begin{equation*}
    \phi: X_d\rightarrow \mathbb{P}^n,\quad\quad [t:x_0:\dotsi:x_n]\mapsto [x_0:\dotsi:x_n].
\end{equation*}
It is branched along the smooth hypersurface
\begin{equation*}
H_{\epsilon}=V(x_0\dots x_{d-1} -\epsilon f(x_0,\dots,x_n))\subseteq \mathbb{P}^n  .  
\end{equation*}

Let $\omega_{\FS}$ be the Fubini-Study K\"ahler form on $\mathbb{P}^n$, scaled so that
\begin{equation*}
    c_1(\mathbb{P}^n) = (n+1)[\omega_{\FS}].
\end{equation*}
The symplectic manifold $(\mathbb{P}^n,\omega_{\FS})$ admits a Hamiltonian action by $(S^1)^n$ that is free in the complement of a union of $n+1$ hyperplanes,
\begin{equation*}
   P = \bigcup_{i=0}^n \{ x_i=0\}.
\end{equation*}
The moment map of this Hamiltonian action is
\begin{align}\label{moment map}
    M:\mathbb{P}^n &\rightarrow \Delta\\
    [x_0:\dotsi:x_n] &\mapsto \left(\frac{\abs{x_0}^2}{\abs{x}^2},\dots, \frac{\abs{x_n}^2}{\abs{x}^2}\right),\notag
\end{align}
where $\abs{x}^2=\abs{x_0}^2+\dotsi+\abs{x_n}^2$ and $\Delta\subseteq \mathbb{R}^{n+1}$ is the $n$-dimensional simplex,
\begin{equation*}
  \Delta = \left\{ (r_0,\dots,r_n)\in (\mathbb{R}_{\geq 0})^{n+1} \ | \ \sum_{i=0}^n r_i = 1 \right\}. 
\end{equation*}

The fibers of the moment map over the interior of the simplex are Lagrangian tori parametrized by $\mathbf{r}\in \text{int}(\Delta)$:
\begin{equation*}
    L_{\mathbf{r}}=\{[x_0:\dots:x_n] \ | \ r_0^{-1}\abs{x_0}=\dots=r_n^{-1}\abs{x_n} \}.
\end{equation*}
There is a convenient generating set of the homology group $H_2(\mathbb{P}^n,L_{\mathbf{r}},\mathbb{Z})$. It is given by holomorphic discs $v_0,\dots,v_n$, where
\begin{equation}\label{generators of relative H2}
    v_k(\mathbf{r}) : \disk \rightarrow (\mathbb{P}^n,L_{\mathbf{r}}),\quad
    z \mapsto [r_0:\dots:r_{k}z:\dots:r_{n}].
\end{equation}
These classes add up to the spherical class that generates $H_2(\mathbb{P}^n,\mathbb{Z})$. Moreover, for each $k=0,\dots,n$,
\begin{equation*}
    \mu_{L_\mathbf{r}}(v_k)=2\quad\text{and}\quad\text{Area}_{\omega_{\FS}}(v_k)=r_k.
\end{equation*}

The torus $L_\cl$ corresponding to $\mathbf{r}=(1/(n+1),\dots,1/(n+1))$ is usually called the Clifford torus. It is the only one among these tori that is monotone with respect to the Fubini-Study K\"ahler form. Consider instead the Lagrangian torus $L_\mathbf{r}\subseteq\mathbb{P}^n$ corresponding to
\begin{equation*}
    r_0=\dots=r_{d-1}=\frac{1}{d(n+2-d)}\quad\text{and}\quad r_d=\dots =r_n=\frac{1}{n+2-d}.
\end{equation*}
When $\epsilon$ is sufficiently small, $L_\mathbf{r}$ is disjoint from $H_{\epsilon}$ and therefore its pre-image $L\subseteq X_d$ is a connected Lagrangian torus, see Proposition \ref{connected pre-image}. Let $\omega$ be the K\"ahler form on $X_d$ that is constructed in Proposition \ref{Kahler form}.
\begin{lemma}\label{lift is monotone}
The Lagrangian torus $L\subseteq (X_d,\omega)$ is monotone.
\end{lemma}
\begin{proof}
Consider the following linear map
\begin{equation*}
   \ell: H_2(\mathbb{P}^n,L_{\mathbf{r}},\mathbb{Z})\rightarrow \mathbb{Q}, \quad v\mapsto \frac{1}{2}\mu_{L_{\mathbf{r}}}(v)-\frac{d-1}{d}v.H_{\epsilon}.
\end{equation*}
Since $H_{\epsilon}$ is a small deformation of $\{x_0\dotsi x_{d-1}=0\}$, we have that $\ell(v_k)=1/d$ in the range $0\leq k\leq d-1$ and that $\ell(v_k)=1$ in the range $d\leq k\leq n$. It follows that for all $v\in H_2(\mathbb{P}^n,L_{\mathbf{r}})$,
\begin{equation*}
    \text{Area}_{\omega_{\FS}}(v) = \frac{\ell(v)}{n+2-d}.
\end{equation*}
Let $u\in H_2(X_d,L,\mathbb{Z})$ be a disc class and let $v=\phi_*(u)$ be its pushforward. By the Riemann-Hurwitz formula (\ref{Riemann Hurwitz}),
\begin{align*}
    \frac{1}{2}\mu_L(u)=\ell(v)&=(n+2-d)\text{Area}_{\omega_\FS}(v)\\
    &=(n+2-d)\text{Area}_{\omega}(u).
\end{align*}
It follows that $L\subseteq (X_d,\omega)$ is  monotone.
\end{proof}
By Theorem \ref{super-potential-theorem}, the super-potential associated with the monotone Lagrangian torus $L\subseteq X_d$ is given by
\begin{equation}\label{super-potential for hypersurface}
   W_L = (z_0+\dotsi+z_{d-1})^d+z_d+\dotsi+z_n.
\end{equation}
It should be viewed as a function on the variety
\begin{equation}\label{mirror-space}
    \spec(\mathbb{C}[H_1(L,\mathbb{Z})]) = \{(z_0,\dots,z_n)\in\mathbb{C}^{n+1}\ |\ z_0\cdots z_n=1\}/\mathbb{Z}_d,
\end{equation}
where the action of $\mathbb{Z}_d$ is by $d$-roots of unity
\begin{equation*}
    \zeta\cdot(z_0,\dots,z_n) = (\zeta z_0,\dots,\zeta z_{d-1},z_d,\dotsi,z_n).
\end{equation*}
Since this action is fixed-point free, we can actually compute critical points and critical values of $W_L$ before passing to the quotient.
\begin{remark}
The super-potential $W_L$ from (\ref{super-potential for hypersurface}) agrees with  predictions from closed-string mirror symmetry going back to Batyrev, Hori-Vafa, and Givental. Indeed, the mirror space (\ref{mirror-space}) has global coordinates given by
\begin{align*}
   y_k&= z_k/ z_0 \ \quad\hbox{if } 1\leq k\leq d-1,\\
   x_{k-d}&= z_{k} \quad\ \ \hbox{if } d\leq k\leq n.
\end{align*}
In these coordinates (compare with \cite[\S 3.2]{przyj-formula-for-fano-hypersurface}),
\begin{equation*}
    W_L = x_0+\dotsi+x_{n-d}+\frac{(1+y_1+\dotsi+y_{d-1})^d}{x_0\dotsi x_{n-d}y_1\dotsi y_{d-1}},\quad\hbox{ for } d\leq n.
\end{equation*}
In the index $1$ Fano case, corresponding to $d=n+1$, this formula must be shifted by $-(n+1)!$, which is the regularized gravitational descendant appearing in the super-potential formula (\ref{super-potential-formula}).
\end{remark}
\begin{lemma}
The critical values of the super-potential are $(n+2-d)\lambda$, where $\lambda$ ranges over the solutions of the equation $\lambda^{n+2-d}=d^d$. Each of these critical values is non-degenerate.
\end{lemma}
\begin{proof}
Let $f=z_0\dotsi z_n -1$. A point $\mathbf{z}=(z_0,\dots,z_n)$ is critical for $W_L$ if
\begin{equation}\label{proportional-gradients}
    \nabla_\mathbf{z} W_L = \lambda \nabla_\mathbf{z} f\quad \hbox{for some }\lambda\in\mathbb{C}.
\end{equation}
Let $z=(z_0+\dotsi+z_{d-1})/d$. Then, equation (\ref{proportional-gradients}) implies that
\begin{align*}
   d^dz^{d-1}&=\lambda/z_k \quad\hbox{if } 0\leq k\leq d-1,\\
   z_k&=\lambda \quad\quad\ \hbox{if } d\leq k\leq n.
\end{align*}
 In particular $z_0=z_1=\dots=z_{d-1}=z$, and $\lambda=d^dz^d$. Since $z_0\dotsi z_n =1$, we see that $z^d\times \lambda^{n-d+1}=1$, and therefore that $\lambda^{n-d+2}=d^d$. The corresponding critical value is $d^dz^d+(n-d+1)\lambda = (n+2-d)\lambda $.

Let $\mathbf{z}$ be a critical point corresponding to a pair $(\lambda,z)$ such that
\begin{equation*}
    \lambda^{n-d+2}=d^d\quad\text{and}\quad z^d =\lambda/d^d.
\end{equation*}
To show it is non-degenerate, we change to more convenient coordinates using the transformation
\begin{align*}
   z_k&\mapsto z\cdot z_k \ \quad\hbox{if } 0\leq k\leq d-1,\\
   z_k&\mapsto \lambda\cdot z_k \quad\ \hbox{if } d\leq k\leq n.
\end{align*}
The critical point then becomes $\mathbf{z}=(1,\dots,1)$, and
\begin{equation*}
    \frac{1}{\lambda}W=\frac{1}{d^d}(z_0+\dotsi+z_{d-1})^d+z_d+\dotsi+z_{n-1}+\frac{1}{z_0\dotsi z_{n-1}}.
\end{equation*}
The Hessian is computed to be the block matrix
\begin{equation*}
\text{Hess}_{\mathbf{z}}(W/\lambda) =
\left(\begin{array}{@{}c|c@{}}
  \begin{matrix}
  c\mathbf{1}_d+I
  \end{matrix}
  & \mathbf{1} \\
\hline
  \mathbf{1} &
  \begin{matrix}
  \mathbf{1}_{n-d}+I
  \end{matrix}
\end{array}\right),    
\end{equation*}
where $c=2-1/d$, and $\mathbf{1}_d$ is the $d\times d$ matrix whose coefficients are all $1$. The determinant of this matrix is $(-1)^nd(n-d+2)$, so it is invertible.
\end{proof}
\begin{remark}
The critical values above are also the \emph{small} eigenvalues of the map
\begin{equation}\label{quantum-multiplication}
    QH(X_d)\rightarrow QH(X_d):\quad \alpha \mapsto c_1(X_d)\star \alpha.
\end{equation}
There is one more eigenvalue, the \emph{big} eigenvalue, which is $0$. It is not seen by the Lagrangian torus $L\subseteq X_d$ when $d<n+1$.
\end{remark}
\begin{corollary}
For each \emph{small} eigenvalue $\lambda$ of (\ref{quantum-multiplication}), there is $\mathbb{C}^*$-local system $\xi_\lambda$ on $L$ such that $(L,\xi_{\lambda})$ split-generates the component $\Fuk(X_d)_{\lambda}$ of the Fukaya category.
\end{corollary}
\begin{proof}
See Proposition \ref{split-generation} in the Appendix.
\end{proof}
\subsection{The degree 4 delPezzo} We now study an example of a non-cyclic $4$-fold cover. Consider the delPezzo surface $X=\text{Bl}_5(\mathbb{P}^2)$. It admits a $4$-fold covering map to $\mathbb{P}^2$; the latter has pseudo-index $j_{\mathbb{P}^2}=3$. Such a map can be seen as a composition of $2$-fold cyclic coverings
\begin{equation*}
    \phi:\text{Bl}_5(\mathbb{P}^2)\xrightarrow{\phi_1} Q_2\xrightarrow{\phi_0} \mathbb{P}^2,
\end{equation*}
where $\phi_{0}:Q_2\rightarrow \mathbb{P}^2$ is a two dimensional quadric branched along a conic curve $C\subseteq \mathbb{P}^2$. The double covering $\text{Bl}_5(\mathbb{P}^2)\xrightarrow{\phi_1} Q_2$ in turn is branched along an elliptic curve $E\subseteq Q_2$ as we shall explain below.

Let $\omega_{\text{FS}}$ be the Fubini-Study metric on $\mathbb{P}^2$, scaled so that 
\begin{equation*}
    c_1(\mathbb{P}^2) = 3[\omega_{\text{FS}}] \quad \text{in}\ H_{\text{dR}}^2(\mathbb{P}^2,\mathbb{R}).
\end{equation*}
Then, 
\begin{equation}
    L = \{[x_0:x_1:x_2]\in \mathbb{P}^2\ |\ \abs{x_0}=2\abs{x_1}=2\abs{x_2}\}\subseteq (\mathbb{P}^2,\omega_{\FS})
\end{equation}
is a Maslov positive Lagrangian torus. Up to automorphism, it bounds exactly three discs $u_k:\disk\rightarrow(\mathbb{P}^2,L)$ of Maslov index $2$ through the point $[2:1:1]$, each of them is Fredholm regular:
\begin{gather}\label{Clifford-discs}
    u_0(z) = [2z:1:1]; \quad u_1(z)=[2:z:1]; \quad u_2(z)=[2:1:z]\\ 
    \text{Area}_{\omega_{\FS}}(u_0)=1/2; \quad \text{Area}_{\omega_{\FS}}(u_1)=1/4; \quad \text{Area}_{\omega_{\FS}}(u_2)=1/4.\notag
\end{gather}
This collection of discs generates $H_2(\mathbb{P}^2,L,\mathbb{Z})$. The super-potential is
\begin{equation*}
W_{L}= z_0+z_1+z_2,\quad \hbox{ where } z_k=z_{\partial u_k}\in \mathbb{c}[H_1(L,\mathbb{Z})].  
\end{equation*}
Note that $z_0z_1z_2=1$. We choose a conic $C\subseteq \mathbb{P}^2$ of the form
\begin{equation*}
    C_{\epsilon} = \{ x_1x_2-\epsilon x_0^2 =0\}\subseteq\mathbb{P}^2, \quad \abs{\epsilon}<1.
\end{equation*}
The quadric that is branched along $C_{\epsilon}$ is given by the equation
\begin{equation*}
    Q_2=\{t^2=x_1x_2-\epsilon x_0^2\}\subseteq \mathbb{P}^3.
\end{equation*}
By Theorem \ref{super-potential-theorem}, there is a K\"ahler form $\widehat{\omega}$ on $Q_2$ such that the pre-image $\widehat{L}\defeq\phi_0^{-1}(L)\subseteq (Q_2,\widehat{\omega})$ is a Lagrangian torus and
\begin{equation}\label{potential for L-hat}
    (\phi_0)_*W_{\widehat{L}}=z_0 + (z_1+z_2)^2.
\end{equation}
The classes that lift to Maslov index $2$ holomorphic discs in $(Q_2,\widehat{L})$ are
\begin{equation}\label{classes that lift}
    [u_0],\ 2[u_1],\ 2[u_2],\hbox{ and }[u_1+u_2] \quad \hbox{in }H_2(\mathbb{P}^2,L).
\end{equation}
Their lifts in fact generate $H_2(Q_2,\widehat{L},\mathbb{Q})$.
Using the area computation in (\ref{Clifford-discs}), one verifies that $\widehat{L}\subseteq (Q_2,\widehat{\omega})$ is monotone.

In the embedding $Q_2\subseteq \mathbb{P}^3$, we have $c_1(Q_2)=\mathscr{O}_{Q_2}(2)$. Therefore, A smooth anti-canonical divisor is an elliptic curve $E$ given as the intersection of a second quadric $Q_2'$ with $Q_2$. Consider the following choice
\begin{equation*}
    Q_2' = \{tx_0=\epsilon f(x_0,x_1,x_2)\}, \quad \abs{\epsilon} \hbox{ small},\ f\hbox{ generic}.
\end{equation*}
Then, setting $E=Q_2\cap Q_2'$, we see that
\begin{equation*}
    \phi_0(E) = \{x_1x_2x_0^2-\epsilon x_0^4-\epsilon^2 f(x_0,x_1,x_2)^2=0\}\subseteq \mathbb{P}^2,
\end{equation*}
which is a small deformation of the union of lines $\Delta_0=\{x_1x_2x_0^2=0\}\subseteq\mathbb{P}^2$. This singular quartic is disjoint from the torus $L\subseteq \mathbb{P}^2$ and its intersection number with each one of the classes that lift (see (\ref{classes that lift})) is $2$. It follows that when $\abs{\epsilon}$ is small, the monotone Lagrangian torus $\widehat{L}$ is disjoint from the elliptic curve $E$, and has a vanishing Maslov class in its complement.

The $2$-fold covering of $Q_2$ branched along $E$ is again a Fano surface $S$. It's Euler characteristic is
\begin{equation*}
    \chi(S) = 2\chi(Q_2\backslash E) + \chi(E) = 2\chi(Q_2)-\chi(E)=8.
\end{equation*}
Hence $S\cong\text{Bl}_5(\mathbb{P}^2)$. By Theorem \ref{super-potential-theorem} and Corollary \ref{monotone case}, there is an anti-canonical K\"ahler form $\omega_S$ on $S$ and a monotone Lagrangian torus $L^+\subseteq (S,\omega_S)$ whose super-potential is given by
\begin{equation}\label{potential for L+}
    \phi_* W_{L^+} = (z_0 + (z_1+z_2)^2)^2-4.
\end{equation}
In order to compare our formulas with the literature (see \cite{galkin-mutations} and \cite[Table 1]{pascaleff-tonkonog}), it is instructive to write the super-potentials (\ref{potential for L-hat}) and (\ref{potential for L+}) as Laurent polynomials. The inclusion
\begin{equation*}
    \mathbb{C}[H_1(\widehat{L},\mathbb{Z})] \hookrightarrow \mathbb{C}[H_1(L,\mathbb{Z})]
\end{equation*}
is generated by the monomials $x=z_0$ and $y=z_1z_2^{-1}$. In these coordinates,
\begin{equation*}
    W_{\widehat{L}} = x+\frac{1}{xy}(1+y)^2.
\end{equation*}
The quadric $Q_2$ is isomorphic to $\mathbb{P}^1\times\mathbb{P}^1$ whose toric super-potential is (see \cite{galkin-mutations})
\begin{equation}\label{super potential of p1 times p1}
    W_{\mathbb{P}^1\times\mathbb{P}^1} = x+y+\frac{1}{x}+\frac{1}{y}.
\end{equation}
Our formula is related to the one above by the following birational transformation of $(\mathbb{C}^*)^2$,
\begin{equation*}
    (x,y)\mapsto (x/(1+y),xy/(1+y)).
\end{equation*}
Such a birational transformation is called a \emph{mutation}, see \cite{katzarkov-przyj-old-and-new} and \cite{acgk-mutations} for the relevant literature.
Similarily, the inclusion
\begin{equation*}
    \mathbb{C}[H_1(L^+,\mathbb{Z})]\hookrightarrow \mathbb{C}[H_1(L,\mathbb{Z})]
\end{equation*}
is generated by $x=z_0^2$ and $y=z_0 z_1^2$. In these coordinates,
\begin{equation*}
    W_{L^+} = \frac{\left(x+\frac{1}{y}(1+y)^2\right)^2}{x}.
\end{equation*}
The super-potential computed for the delPezzo surface $S=\text{Bl}_5(\mathbb{P}^2)$ in \cite{galkin-mutations} is
\begin{equation}\label{super potential of delPezzo}
    W_S = \frac{(1+x)^2(1+y)^2}{xy}-4.
\end{equation}
Again, it is related to our computation of $W_{L^+}$ through the \emph{mutation}
\begin{equation*}
    (x,y)\mapsto (xy/(1+y)^2,y).
\end{equation*}
\begin{remark}
There is a more direct construction of a monotone Lagrangian torus $L^-\subseteq \text{Bl}_5(\mathbb{P}^2)$. Indeed, the toric variety $\mathbb{P}^1\times\mathbb{P}^1$ contains a monotone Lagrangian torus $\overline{L}$ arising as the central fiber of its moment map. The super-potential of $\overline{L}$ is (\ref{super potential of p1 times p1}).
We can construct a double covering of $\mathbb{P}^1\times\mathbb{P}^1$ by branching along a smoothing of the toric boundary. This double covering is again $\text{Bl}_5(\mathbb{P}^2)$. The pre-image $L^-$ of $\overline{L}$ is a monotone Lagrangian torus and its super-potential is (\ref{super potential of delPezzo}).
\end{remark}
\subsection{The second Hirzebruch surface} We now present an example of a non-Fano manifold where Theorem \ref{main theorem} still applies. The second Hirzebruch surface $\mathbb{F}_2=\mathbb{P}(\mathscr{O}_{\mathbb{P}^1}\oplus \mathscr{O}_{\mathbb{P}^1}(-2))$ is a toric surface whose toric super-potential is
\begin{equation}\label{toric-super-potential}
    W= x+y+\frac{1}{x}+\frac{1}{x^2y}.
\end{equation}
The surface $\mathbb{F}_2$ can also be obtained from $\mathbb{P}^2$ by taking a $2$-fold covering branched along the divisor $C=\{x_0x_1=0\}$ and blowing up to resolve the $A_1$ singularity that arises. Let $\phi:\mathbb{F}_2\rightarrow\mathbb{P}^2$ be the resulting map. Note that $\phi^{-1}([0:0:1])=E$ is the $-2$ curve in $\mathbb{F}_2$. Let $\Delta=\{x_0x_1x_2=0\}\subseteq \mathbb{P}^2$ be the toric boundary. The reduced pre-image $D=\phi^{-1}(\Delta)^{\text{red}}$ is an snc anti-canonical divisor with $4$ components, one of which is the $-2$ curve $E$. The Clifford torus $L_{\cl}\subseteq \mathbb{P}^2\backslash\Delta$ lifts to an admissible torus $L=\phi^{-1}(L_\cl)\subseteq \mathbb{F}_2\backslash D$. Although $\mathbb{F}_2$ is not Fano ($c_1(E)=0$), one can still define the super-potential $W_L$ the same way as before (i.e. Definition \ref{super-potential-definition}), by deforming the complex structure away from $D$. The only difference is that we need to explicitly rule out spherical bubbles with vanishing Chern number from the boundary of moduli spaces of Maslov index $2$ discs.

\begin{lemma}
Suppose $\alpha\in H_2(\mathbb{F}_2,L)$ is a Maslov index $2$ class which splits as $\alpha=A+\tilde{\alpha}$ for some spherical class $A$ such that $c_1(A)=0$. Let $J$ be an almost complex structure which preserves $D$. If $A$ supports a $J$-holomorphic sphere, and $\alpha$ and $\tilde{\alpha}$ support $J$-holomorphic discs, then $A=0$.
\end{lemma}
\begin{proof}
Let $(D_i)_{i=1}^4$ be the components of the snc divisor $D\subseteq\mathbb{F}_2$. Since $L$ is admissible, we have $\alpha\cdot D = \tilde{\alpha}\cdot D = 1$. Therefore, there are distinct $i,j\in\{1,2,3,4\}$ such that $\alpha\cdot D_i=\tilde{\alpha}\cdot D_i=0$ and $\alpha\cdot D_j=\tilde{\alpha}\cdot D_j=0$. It follows (by positivity of intersection) that the spherical bubble supported in $A$ lies at the intersection $D_i\cap D_j$, so it must be constant. 
\end{proof}


Recall that the super-potential associated with $L_{\cl}\subseteq \mathbb{P}^2$ is
\begin{equation*}
    W_{L_\cl}= x+y+\frac{1}{xy},
\end{equation*}
and that $W_{L_\cl}^C = x+y$. Using Theorem \ref{super-potential-theorem}, we deduce that
\begin{equation*}
    W_L = x^2+xy+y^2+\frac{1}{xy} \quad \text{on}\ (\mathbb{C}^*)^2/\mathbb{Z}_2.
\end{equation*}
Using the coordinates $u=(xy)^{-1}$ and $v=x^2$ on the quotient, one sees that $W_L$ is exactly the toric super-potential (\ref{toric-super-potential}) of $\mathbb{F}_2$. 

We warn the reader that our discussion above only accounts for the contributions of \emph{smooth} discs to the super-potential. The full mirror to $\mathbb{F}_2$ has an extra term coming from contributions of \emph{nodal} discs which we are completely ignoring. These contributions and their relevance to mirror symmetry are further discussed in \cite[\S 3.2]{Auroux-F2}.

\section{Appendix}
The purpose of this appendix is to collect results from the literature leading up to the proof of the split-generation result stated in Proposition \ref{split-generation}.
\subsection{Notation} 
Let $(X,\omega)$ be a closed symplectic manifold such that
\begin{equation*}
[\omega]=c_1(X).    
\end{equation*}
\begin{definition}
A Lagrangian brane $\mathbf{L}=(L,P,\xi)$ is a closed oriented monotone Lagrangian submanifold $L\subseteq X$ which is equipped with a spin structure $P$ and a $\mathbb{C}^*$-local system $\xi$. 
\end{definition}
We have already encountered an invariant for Lagrangian branes,
\begin{equation*}
    m_0(\mathbf{L}) = \sum_{\mu_L(\beta)=2} m_{0,\beta}(L,P)\hol_{\xi}(\partial \beta).
\end{equation*}
We will often omit the spin structure $P$ in our notation.
Given $\lambda\in\mathbb{C}$, the collection of Lagrangian branes that satisfy the equation $m_0(\mathbf{L})=\lambda$ are the objects of the $\lambda$-component of Fukaya's $A_{\infty}$-category $\Fuk(X)_{\lambda}$; one usually restricts to a finite or countable collection $(\mathbf{L}_i)$. For each pair $(i,j)$, we choose a Hamiltonian $H^{ij}:X\rightarrow \mathbb{R}$ and a time-dependent almost complex structure $J_t^{ij}$. The pair $(H,J_t)=(H^{ij},J_t^{ij})$ is called a Floer datum. A Floer trajectory for the pair $(L_i,L_j)$ is a map $u:\mathbb{R}\times [0,1]\rightarrow X$ such that $u(-,0)\in L_i$,  $u(-,1)\in L_j$,
\begin{gather}\label{Floer trajectories}
    \partial_s u + J_t (\partial_t u - X_H(u)) = 0, \hbox{ and }\ 
    \int_Z \abs{du}^2< +\infty.
\end{gather}
We follow the conventions of \cite{PL}, where the Hamiltonian vector field $X_H$ solves the equation $\iota_{X_H}\omega = -dH$.
The first step towards building Fukaya's $A_{\infty}$-category is to choose perturbation data $(H^{ij},J_t^{ij})$ for each pair $(i,j)$ such that
$\varphi_1^{H^{ij}}(L_i)\cap L_j$ is transverse, and solutions to (\ref{Floer trajectories}) are Fredholm regular.

Given two objects $\mathbf{L}_0,\mathbf{L_1}\in \Fuk(X)_{\lambda}$, the hom-space $(CF^*(\mathbf{L}_0,\mathbf{L}_1),\mu^1)$ is a $\mathbb{Z}_2$-graded chain complex over the field of complex numbers .
Let $\mathscr{C}(H)$ be the set of Hamiltonian chords for $H=H^{01}$. That is, paths $\gamma:[0,1]\rightarrow X$ such that $\gamma(0)\in L_0$, $\gamma(1)\in L_1$, and
\begin{equation*}
    \partial_t\gamma(t) = X_{H}(\gamma(t)).
\end{equation*}
Floer's complex is the vector space
\begin{equation*}
    CF^*(\mathbf{L}_0,\mathbf{L}_1) = \bigoplus_{y\in\mathscr{C}(H)} \mathbb{C}\langle y\rangle,
\end{equation*}
where the differential is given by
\begin{equation*}
    \mu^1(y_1) = \sum_{\text{ind}(u)=1} \text{sgn}(u)\hol(\partial u)y_0.
\end{equation*}
The sum above is taken over isolated points of the space $\matheu{M}_Z(y_0,y_1)$ of \emph{unparametrized} Floer trajectories (\ref{Floer trajectories}) such that
\begin{gather*}
 u(-\infty,-)=y_0\quad\hbox{and}\quad u(+\infty,-)=y_1.
\end{gather*}
The integer $\text{sgn}(u)\in\{-1,1\}$ is a sign computed using the spin structures on the Lagrangian branes, and $\hol(\partial u)$ is a holonomy factor computed using their associated local systems. Higher operations $(\mu^d)_{d\geq 2}$ are constructed similarly, see \cite{PL,ritter-smith,sheridan-fano} for the full construction. Given $\mathbf{L}\in \Fuk(X)_\lambda$, we set $HF^*(\mathbf{L})$ to be the cohomology of $\mu^1$ for the pair $(\mathbf{L},\mathbf{L})$. This is a $\mathbb{Z}_2$-graded associative algebra over $\mathbb{C}$ with a unit.

\subsection{Generation} Let $L$ be an oriented monotone Lagrangian torus which is equipped with a spin structure. An element $b\in H^1(L,\mathbb{C})$ determines a local system
\begin{equation*}
    \xi_b:\pi_1(L)\rightarrow \mathbb{C}^*,\quad [\gamma]\mapsto \exp(\gamma\cdot b). 
\end{equation*}
It also determines an evaluation map
\begin{equation*}
    ev_b: \mathbb{C}[H_1(L,\mathbb{Z})]\rightarrow \mathbb{C},\quad z_{\beta}\mapsto \langle\partial\beta,b \rangle.
\end{equation*}
Its kernel is a maximal ideal $\mathbf{m}_b\unlhd \mathbb{C}[H_1(L,\mathbb{Z})]$, and in fact,
\begin{equation*}
    m_0(L,\xi_b) = W(\mathbf{m}_b).
\end{equation*}
All maximal ideals of $\mathbb{C}[H_1(L,\mathbb{Z})]$ arise in this way.
The remainder of the appendix is dedicated to proving the following result.
\begin{proposition}\label{split-generation}
If $\lambda\in\mathbb{C}$ is a critical value of $W:\mathbb{C}^*\rightarrow \mathbb{C}$, then it is an eigenvalue of the quantum multiplication map
\begin{equation*}
    QH(X)\rightarrow QH(X):\quad  A\mapsto c_1(X)\star A.
\end{equation*}
Suppose further that the generalized eigenspace $QH(X)_{\lambda}$ is $1$-dimensional. Let $b\in H^1(L,\mathbb{C})$ be an element such that $\mathfrak{m}_b\in W^{-1}(\lambda)$ is a critical point. Then, the object $\mathbf{L}_b=(L,\xi_b)$ split-generates the component $\Fuk(X)_\lambda$ of the monotone Fukaya category.
\end{proposition}

By the generation result in \cite[Theorem 12.19]{ritter-smith}, it is enough to show that $HF(\mathbf{L}_b)\neq 0$. We will in fact show that there is a vector space isomorphism
\begin{equation*}
    HF(\mathbf{L}_b) \cong H^*(L,\mathbb{C}).
\end{equation*}
In the work of Fukaya-Oh-Ohta-Ono \cite{FO3-book1}, the authors construct the structure of an $A_{\infty}$-algebra $(\mathfrak{m}_k)_{k\geq 1}$ on the vector space $C_*(L,\mathbb{C})$ of singular chains on $L$; the monotonicity assumption ensures that we can work over $\mathbb{C}$, see \cite[Theorem 3.1.5]{FO3-book1}. 

Given two Lagrangian branes $\mathbf{L}_0$ and $\mathbf{L_1}$ whose underlying Lagrangians are transverse, the authors also upgrade the construction we outlined above for $CF^*(\mathbf{L}_0,\mathbf{L}_1)$ to an $A_{\infty}$-bimodule $C(\mathbf{L_0},\mathbf{L_1})$, with operations $\mathfrak{n}_{k_0,k_1}$ over the $A_{\infty}$-algebras $(C_*(L_0,\mathbb{C}),\mathfrak{m}_k^{\mathbf{L}_0})$ and $(C_*(L_1,\mathbb{C}),\mathfrak{m}_k^{\mathbf{L}_1})$, see \cite[Theorem 3.7.21]{FO3-book1}. We emphasize that when $L_0$ and $L_1$ are transverse, we have an identity of complexes
\begin{equation}\label{Floer 1 is Floer 2}
    (C(\mathbf{L_0},\mathbf{L_1}),\mathfrak{n}_{0,0}) = (CF^*(\mathbf{L_0},\mathbf{L_1}),\mu^1),
\end{equation}
provided one uses the same almost complex structure $(J_t)$, and the trivial Floer datum $H=0$: this is sufficient to achieve transversality for (\ref{Floer trajectories}).

More generally, the $A_{\infty}$-bimodule $(C(\mathbf{L_0},\mathbf{L_1}),\mathfrak{n}_{k_0,k_1})$ is constructed under the assumption that $L_0$ and $L_1$ intersect \emph{cleanly}, see \cite[Definition 3.7.48]{FO3-book1}. In the case when $\mathbf{L_0}=\mathbf{L_1}=\mathbf{L}$ (this is a clean intersection), one can make appropriate choices (see \cite[Proposition 3.7.73]{FO3-book1}) so that
\begin{equation}\label{bimodule is algebra}
    \mathfrak{n}_{k_1,k_0} = \mathfrak{m}^{\mathbf{L}}_{k_1+1+k_0}.
\end{equation}
Finally, it is shown in \cite[Theorem 5.3.14]{FO3-book1} that if $H$ is a Hamiltonian such that  both $L_0\cap L_1$ and $\varphi_1^{H}(L_0)\cap L_1$ are clean intersections, then there is a homotopy equivalence of $A_{\infty}$-bimodules
\begin{equation}\label{homotopy equivalence}
    \mathfrak{h}: (C(\mathbf{L_0},\mathbf{L_1}),\mathfrak{n}_{k_1,k_2})\rightarrow (C(\varphi_1^{H}(\mathbf{L_0}),\mathbf{L_1}),\mathfrak{n}'_{k_1,k_2}).
\end{equation}
The following is a direct consequence of (\ref{Floer 1 is Floer 2}), (\ref{bimodule is algebra}), and (\ref{homotopy equivalence}).
\begin{lemma}\label{morse bott = floer trajectories}
Let $\mathbf{L}$ be a Lagrangian brane. Then, as vector spaces over $\mathbb{C}$,
\begin{equation*}
    H((C_*(L,\mathbb{C}),\mathfrak{m}_1^{\mathbf{L}}))\cong HF^*(\mathbf{L}).
\end{equation*}
\end{lemma}

The $A_{\infty}$-algebra $(C_*(L,\mathbb{C}),\mathfrak{m}_k^{\mathbf{L}})$ is more computable because it satisfies a divisor axiom. In contrast, $HF^*(\mathbf{L})$ can be defined over $\mathbb{Z}$ and as such it cannot satisfy a divisor axiom.
\begin{lemma}(\cite[Lemma 13.1]{adic-convergence})\label{divisor axiom}
The $A_{\infty}$-algebra $(C_*(L,\mathbb{C}),\mathfrak{m}_k^{\mathbf{L}})$ has a canonical model $(H^*(L,\mathbb{C}),\mathfrak{m}_k)$. Each term $\mathfrak{m}_k$ has a decomposition
\begin{equation*}
    \mathfrak{m}_{k} = \sum_{\beta\in H_2(X,L)} \mathfrak{m}^L_{k,\beta}\emph{\hol}_{\xi}(\partial \beta).
\end{equation*}
Moreover, for all $s\geq 0$, $k\geq 0$, and $b,x_1,\dots,x_k\in H^1(L,\mathbb{C})$,
\begin{equation*}
    \sum_{s_0+\dots+s_k=s} \mathfrak{m}^L_{k+s,\beta}(b^{\otimes s_0},x_1,b^{\otimes s_1},\dots,x_k,b^{\otimes s_k}) = \frac{\left(\partial\beta\cap b\right)^s}{s!} \mathfrak{m}^L_{k,\beta}(x_1,\dots,x_k).
\end{equation*}
\end{lemma}
We note that in the previous lemma, the operation $\mathfrak{m}^L_{k,\beta}$ does not depend on the associated local system. 

\textit{Proof of Proposition \ref{split-generation}.}
For each $b\in H^1(L,\mathbb{C})$, we denote by $\mathfrak{m}_k^b$ the $A_{\infty}$-operations of Lemma \ref{divisor axiom} associated with the Lagrangian brane $(L,P,\xi_b)$, and we consider its $q$-deformation
\begin{equation*}
    \widehat{\mathfrak{m}}^b_k = \sum_{\beta\in H_2(X,L)} q^{\omega(\beta)} \mathfrak{m}^b_{k,\beta}.
\end{equation*}
These operations define an $A_{\infty}$-structure on $ H^*(L,\mathbb{C}[[q]])$, where $\mathbb{C}[[q]]$ is the ring of formal power series on $q$. The quotient operations
\begin{equation*}
    (H^*(L,\mathbb{C}[[q]])\otimes \mathbb{C}[[q]]/(q), \widehat{\mathfrak{m}}^b_1,\widehat{\mathfrak{m}}^b_2)
\end{equation*}
obtained by setting $q=0$ recover the singular cohomology of $L$ with its cup-product (up to sign). Because $L$ is a torus, this is an exterior algebra on $H^1(L,\mathbb{C})$.
Since $(q)$ is the unique maximal ideal of $\mathbb{C}[[q]]$, we deduce (by Nakayama's lemma) that the iterated multiplication map
\begin{equation}\label{iterated multiplication}
  \widehat{\mathfrak{m}}^b_2: \bigoplus_{l\geq 2} H^1(L,\mathbb{C}[[q]]) ^{\otimes l}\rightarrow H^*(L,\mathbb{C}[[q]])
\end{equation}
is surjective. Using the divisor axiom, one directly computes that
\begin{gather*}
    \widehat{\mathfrak{m}}^b_0 = q\mathbf{m}_0(L,\xi_b) = qW(\mathbf{m}_b).\\
    \widehat{\mathfrak{m}}^b_1(x) = d_b(\widehat{\mathfrak{m}}^b_0)(x),\ \ \  x\in H^1(L,\mathbb{C}).
\end{gather*}
Therefore, if $\xi_b$ is a critical point for $W$, we deduce that $\widehat{\mathfrak{m}}^b_1 = 0$ on $H^1(L,\mathbb{C})$, and subsequently on all of $H^*(L,\mathbb{C}[[q]])$ because (\ref{iterated multiplication}) is surjective. In particular (setting $q=1$, recall that $L$ is monotone), we deduce that $\mathfrak{m}_1^b=0$, and therefore, using Lemma \ref{morse bott = floer trajectories}, we get an isomorphism of complex vector spaces,
\begin{equation*}
    HF^*(\mathbf{L}) = H^*(L,\mathbb{C})\neq 0.
\end{equation*}
The generation result follows now from \cite[Theorem 12.19]{ritter-smith}.\qed

\bibliographystyle{alpha}
\bibliography{bibliography.bib}
\end{document}